\documentclass[a4paper, 11pt, reqno]{amsart}
\usepackage{amsmath}
\usepackage{amsfonts}
\usepackage{amssymb}
\usepackage{amsthm}
\usepackage{mathtools}
\usepackage{hyperref}
\usepackage{indentfirst}
\usepackage[english]{babel}
\usepackage[T1]{fontenc}
\usepackage{graphicx}
\usepackage{tensor} 
\usepackage[colorinlistoftodos]{todonotes}

\usepackage{subfiles}

\newcommand{\re}{\mathbb{R}}
\newcommand{\rt}{\mathbb{R}^{3}}
\newcommand{\sw}{\mathbb{S}^{2}}
\newcommand{\st}{\mathbb{S}^{3}}
\newcommand{\zz}{\mathbb{Z}}
\newcommand{\si}{\Sigma}

\newcommand{\boH}{\mathcal{H}}

\newcommand\too{\longrightarrow}

\newcommand{\tr}{\operatorname{tr}}
\newcommand{\diver}{\operatorname{div}}
\newcommand{\hess}{\operatorname{Hess}}

\newcommand{\curl}{\operatorname{curl}}

\DeclareMathOperator{\Diff}{Diff}
\DeclareMathOperator{\Ric}{Ric}
\DeclareMathOperator{\interior}{int}
\DeclareMathOperator{\spt}{spt}

\newtheorem*{thmA}{Theorem A}
\newtheorem*{thmB}{Theorem B}
\newtheorem*{thmC}{Theorem C}
\newtheorem*{thmD}{Theorem D}
\newtheorem*{thmE}{Theorem E}
\newtheorem{theo}{Theorem}[]
\newtheorem{prop}[theo]{Proposition}
\newtheorem*{pro}{Area-Charge Inequality}
\newtheorem{lemma}[theo]{Lemma}
\newtheorem{definition}[theo]{Definition}

\newtheorem{rem}[theo]{Remark}

\newcommand*\Item[1][]{%
  \ifx\relax#1\relax  \item \else \item[#1] \fi
  \abovedisplayskip=0pt\abovedisplayshortskip=0pt~\vspace*{-\baselineskip}
}

\begin{document}

\title[Min-max minimal surfaces, horizons and electrostatic systems]{Min-max minimal surfaces, horizons and electrostatic systems} 

\author{Tiarlos Cruz}
\address{
Universidade Federal de Alagoas\\
Instituto de Matem\'atica\\
Macei\'o, AL -  57072-970, Brazil}
\email{cicero.cruz@im.ufal.br}
%\thanks{T. Cruz has been partially suported by  CNPq/Brazil grant 311803/2019-9.}

\author{Vanderson Lima}
\address{
  Universidade Federal do Rio Grande do Sul\\
  Instituto de Matem\'atica e Estat\'istica\\
 Porto Alegre, RS - 91509-900, Brazil}
\email{vanderson.lima@ufrgs.br}

\author{ Alexandre de Sousa}
\address{
Universidade Federal de Alagoas\\
Instituto de Matem\'atica\\
Macei\'o, AL -  57072-970, Brazil}
\email{a.asm@protonmail.com}

\begin{abstract}
We present a connection between minimal surfaces of index one and General Relativity. First, we show that for a certain class of electrostatic systems, each of its unstable horizons is the solution of a one-parameter min-max problem for the area functional, in particular it has index one. Combining this with a theorem by Marques and Neves we obtain a uniqueness result for electrostatic systems. We also prove an inequality relating the area and the charge of a minimal surface of index one in a Cauchy data satisfying the Dominant Energy Condition. 
\end{abstract}

\maketitle

\section{Introduction}

The index of a minimal surface, seen as a critical point of the area functional, is a non-negative integer that (roughly speaking) measures the number of distinct deformations that decrease the area to second order. The connection between stable (index zero) minimal surfaces and General Relativity is well known. For instance: the spatial cross sections of the event horizon in static black holes obeying the null energy condition are stable minimal surfaces; Schoen and Yau used the solution of the Plateau problem in their proof of the
{\it Positive Mass Theorem} \cite{SY1}; the {\it Riemannian Penrose inequality} relates the ADM mass and the area of the outermost horizons for a certain class of Cauchy data, \cite{HI,Bray,BrayLee}; a non-exhaustive list of other works is \cite{FG,G,MN,LN,CCE,A,KWY}.

The main goal of this work is to show that minimal surfaces of index one also appear in a natural way in the context of General Relativity. 

As a motivation, consider a Riemannian $3$-manifold $(M^3,g)$, $E \in \mathfrak{X}(M)$ a tangent vector field, and $V\in C^{\infty}(M)$ such that $V > 0$. The Einstein-Maxwell equations with cosmological constant $\Lambda$ for the electrostatic space-time associated to $(M,g,V,E)$ is the following system of equations (see  Subsection \ref{App})
\begin{align}
\hess_{g}V &= V({\Ric_g} -\Lambda g + 2E^{\flat}\otimes E^{\flat} - |E|^2g),\label{sev1}\\
\Delta_{g} V &= (|E|^2 - \Lambda)V,\label{sev2}\\ 
\diver_{g}E &= 0, \ \curl_g (VE) = 0\label{sev3},
\end{align}
where $E^{\flat}$ is the one-form metrically dual to $E$.

\begin{definition}
We say $(M,g,V,E)$ is an electrostatic system if $(M,g)$ is an oriented Riemannian manifold, $V \in C^{\infty}(M)$ and it is not identically zero, $E \in \mathfrak{X}(M)$, and the equations \eqref{sev1}, \eqref{sev2} and \eqref{sev3} are satisfied for some constant $\Lambda \in \re$. Moreover, the system is called complete, if $(M,g)$ is complete.
\end{definition}
 
Examples of such systems are presented in Section \ref{ME}.

In an electrostatic system, the set $\{V = 0\}$ consists of totally geodesic surfaces. Let $\Omega$ be a maximal region where $V > 0$ and consider the associated static space-time. A compact component of $\partial\Omega \subset \{V = 0\}$ has the physical interpretation of being the cross section of a {\it horizon} (see  ~\cite[Section $2.3.2$]{GP} for a discussion about the concept of horizon). In the standard models, some of these surfaces have positive index (see Section \ref{ME}). Our first result is:

%Our first result is inspired by an analogous result in the case of {\it vacuum static systems} (i.e. $E \equiv 0$) due to L. Ambrozio, \cite{A}. 

\begin{thmA}
Let $(M^3,g,V,E)$ be a compact electrostatic system, such that $V^{-1}(0) = \partial M$ and $\Lambda + |E|^2 > 0$. Then:
\begin{enumerate}
\item If $\partial M$ contains an unstable component, then the number of unstable components of $\partial M$ is equal to one and $M$ is simply connected.
\item Each connected component of $\partial M$ is diffeormorphic to a $2$-sphere.
\end{enumerate}
\end{thmA}

The geometric meaning of the condition $\Lambda + |E|^2> 0$ lies in the fact that $R_g = 2\Lambda + 2|E|^2$ (see Lemma \ref{lemma:propert}), so $(M^3,g)$ has positive scalar curvature.

Generically, a minimal surface of index $k > 0$ is the solution of a local $k$-parameter min-max problem for the area functional, see \cite{W}. 
This motivates the following theorem (for a more precise version see Section \ref{MMCU}).

\begin{thmB}
Consider a complete  electrostatic system $(M^3,g,V,E)$, such that $\Lambda + |E|^2> 0$. Let $\Omega_1, \Omega_2$ be connected maximal regions where $V \neq 0$, such that $\overline\Omega_i$ is compact and let $\si = \partial\Omega_1\cap\partial\Omega_2 \subset V^{-1}(0)$ be unstable. Suppose $\partial\Omega_i\setminus \si$ is either empty or strictly stable, for $i=1,2$. Then $\si$ is the solution of a one-parameter min-max problem for the area functional in $(M,g)$ and has index equal to one.
\end{thmB}

The examples to have in mind concerning this result are the system obtained from the round $3$-sphere choosing $\si$ as the equator, and the de-Sitter Reissner-Nordstr\"om system choosing $\si$ as one of the unstable components of $V^{-1}(0)$ (see Section \ref{ME}). Under the assumptions of Theorem B, if $\partial\Omega$ has a degenerate component (section \ref{sec:index}), although we could not prove that $\si$ is the solution of a min-max problem, we were still able to prove that it has index one.

\begin{thmC}
Consider a complete electrostatic system $(M^3,g,V,E)$, such that $\Lambda + |E|^2 > 0$. Let $\Omega_1, \Omega_2$ be connected maximal regions where $V \neq 0$, such that $\overline\Omega_i$ is compact and let $\si = \partial\Omega_1\cap\partial\Omega_2 \subset V^{-1}(0)$ be unstable. Suppose $\partial(\Omega_1\cup\Omega_2)$ is non empty and stable, and at least one of its components is degenerate. Then $\si$ has index one.
\end{thmC}

In the case $E \equiv 0$, Theorem A was proved by L. Ambrozio \cite{A}. It is important to highlight that Theorem B and Theorem C also apply if $E \equiv 0$, and are new results even in this case.

Let $(M^3,g)$ be an oriented Riemannian 3-manifold and let $E \in \mathfrak{X}(M)$. We can consider $(M^3,g,E)$ as part of a Cauchy data for the Einstein-Maxwell equations and the {\it Dominant Energy Condition (DEC)} for the non-electromagnetic fields implies (see Section \ref{Area-Charge ineq})
$$R_g \geq 2\Lambda + 2|E|^2.$$
Let $\si \subset M$ be an orientable closed surface  with unit normal $N$. The charge of $\si$ with respect to $E$ is defined by
\begin{equation}\label{defcharge}
Q(\Sigma) = \frac{1}{4\pi}\int_{\Sigma}\langle E, N\rangle\,d\mu.
\end{equation}

Assuming the DEC holds, there is an inequality relating the area and the charge of closed stable minimal surfaces in $(M,g)$ (see \cite{Gi,DJR}). In particular, the relation holds for the spatial cross-sections of the horizon of an electrostatic black hole. Some of the models we present in section \ref{ME} have a \emph{cosmological horizon}, whose spatial cross sections are minimal surfaces of index one (see subsection \ref{sec:index}). In this case we can prove the following.

\begin{pro}
Consider an oriented Riemannian $3$-manifold $(M^3,g)$, $E \in \mathfrak{X}(M)$ and $\Lambda \in \re$, such that $R_{g} \geq 2\Lambda + 2|E|^2$. Suppose $\si \subset (M,g)$ is an orientable closed minimal surface of index one, with unit normal $N$. Then,
\begin{equation}\label{aci}
\Lambda|\si| + 16\pi^2\frac{Q(\si)^2}{|\si|} \leq 12\pi + 8\pi\Biggl(\frac{g(\si)}{2} - \biggl[\frac{g(\si)}{2}\biggr]\Biggr),
\end{equation}
where $g(\si)$ and $|\si|$ are the genus and the area of $\si$ respectively.
Moreover, the equality in \eqref{aci} holds if, and only if, $\si$ is totally geodesic, $E|_{\si} = aN$, for some constant $a$, $(R_g)|_{\si} \equiv 2\Lambda + 2a^2$, and $g(\si)$ is an even integer.
\end{pro}

Using the previous theorems combined with the results of \cite{BBN} and \cite{MaNe}, we obtain the following uniqueness result for electrostatic systems.

\begin{thmD}
Let $(M^3,g,V,E)$ be an electrostatic system such that $M$ is closed, $\Lambda + |E|^2 > 0$ and $V^{-1}(0) = \si$ is non-empty and connected. Then $\si$ is a separating sphere, and
\begin{equation*}
\Lambda|\si|  \leq 12\pi,
\end{equation*}
with equality if, and only if, $E \equiv 0$ and $(M,g)$ is isometric to the standard sphere of constant sectional curvature $\frac{\Lambda}{3}$.
\end{thmD}

Theorem D is inspired by an analogous result on the case $E \equiv 0$, due to Boucher-Gibbons-Horowitz \cite{BGH} and Shen \cite{Shen}, however, our methods are completely different from the ones of both these works. Using an approach in the spirit of the work of Shen we have obtained the following.

\begin{thmE}
Let $(M,g,V,E)$ be a compact electrostatic system, such that $V^{-1}(0) = \partial M = \cup_{i=1}^{\ell}\si_i$. Suppose $\sup_M |E|^2 < \Lambda$. Then, 
\begin{equation*}
\sum_{i=1}^{\ell}k_i\left(\frac{\Lambda}{3}\,|\si_i| + \frac{16\pi^2 Q(\si_i)^2}{|\si_i|}\right) \leq 4\pi\sum_{i=1}^{\ell}k_i,
\end{equation*}
where $k_i = \big|(\nabla_g V)|_{\si_i}\big|$. Moreover, the equality holds if, and only if, $E \equiv 0$
and $(M,g)$ is isometric to the hemisphere of constant sectional curvature $\frac{\Lambda}{3}$.
\end{thmE}

We remark that the quantities $k_i$ appearing in last theorem are constants (Lemma \ref{lemma:propert}), called of \emph{surface gravities} associated to the horizons.

Let us discuss the proofs and the content of the main results. Theorem A is a generalization of a result in the case of {\it vacuum static systems} (i.e. $E \equiv 0$), due to Ambrozio \cite{A}. A key step in the proof in \cite{A} relies on constructing a singular Einstein four manifold from a static system and prove a result in the spirit of the Bonnet-Myers Theorem. However, it seems to be difficult to adapt this approach to our more general setting. We settle this by borrowing some ideas from Ambrozio's proof and using some results from the topological theory of $3$-manifolds and its connection to minimal surfaces.

In Theorems B and C we use the local {\it Min-Max Theory for the area functional} developed by Ketover, Liokumovich and Song \cite{KLS}. More precisely for Theorem B, we use the conclusions of Theorem A to prove that $\si$ attains the maximum of the area in a \emph{sweepout}, and proceed to show that this surface realizes the {\it min-max width}. In Theorem C we use the same idea modifying the metric in a neighborhood of $\partial\Omega$. We believe Theorems A, B and C may be of relevance to the problem of classification of electrostatic systems.

The study of static systems is a well established topic of research. In the context of General Relativity, those correspond to certain static space-times. Some important works, for instance, in the case of vacuum are \cite{I1,BM,Chr}. For the electrovaccum case some references are \cite{I2,CT,CGa}. From the geometric point of view, to be static vacuum ($E \equiv 0$) means that the formal adjoint of the linearization of the scalar curvature has non-trivial kernel (\cite{FM}), a fact which has applications to the problem of prescribing the scalar curvature (see \cite{Corv}).

The Min-max theory has been a topic of intense research in the last years, mainly due to the work of Marques and Neves.
The discrete setting originated in the work of Almgren \cite{Alm2} and Pitts \cite{P}, with the subsequent work on regularity by Schoen and Simon \cite{SS}. Important recent developments were achieved in \cite{MaNeindexbound,LMN,Zhou3}. These tools led to many applications, as the proof of the {\it Willmore Conjecture} \cite{MaNeWillmore}, the proof of the {\it Freedman-He-Wang Conjecture} \cite{AMN}, the discovery of the density and the equidistribution of minimal hypersurfaces for generic metrics \cite{IMN,MNS}, and the proof by Song \cite{So2} of  {\it Yau's Conjecture} on the existence of infinitely many minimal surfaces (which builds on the earlier work by Marques and Neves \cite{MaNeinfinity}, where they settle the case when the ambient has positive Ricci curvature). 

The continuous setting of the min-max theory initiated in the work of Simon and Smith \cite{Smith} (the basic theory is presented in \cite{C&DL}). Important recent developments were achieved in \cite{Ketgenusbound,KeMaNe,MaNeindexbound}. This version has applications to the theory of Heegaard splittings of $3$-manifolds \cite{CGK,KLS} and to variational geometric problems \cite{MaNe,KLS}. 

%There is also an alternative approach for min-max due to P. Gaspar and M. Guaraco \cite{Ga1,GaGa1}, which is based on the Allen-Cahn equation. Using this version, Chodosh and Mantoulidis gave a proof of the multiplicity one conjecture in dimension $3$, and used this to show the generic version of Yau's Conjecture, \cite{ChMa}. Also, Gaspar and Guaraco obtained new versions of the density and equidistribution results, \cite{GaGa2}.

Inequalities relating the "size" and the physical quantities of black holes (e.g. mass, charge and angular momentum) are a subject of great interest since the beginning of the theory of these objects. Recently, there was an increase in interest in those type of inequalities for other relativistic objects. The recent survey \cite{DG} is a nice and comprehensive introduction to this subject. In the specific case of inequalities relating area and charge, we highlight the works \cite{Gi,DJR,Sim,Kh} which deals with stable minimal surfaces, isoperimetric surfaces, stable MOTS and ordinary objects, respectively.

The paper is organized as follows. In Section \ref{SES} we clarify the physical context where electrostatic systems appear, and we present some general geometric properties of such systems. Next, in Section \ref{ME}, we present several standard model solutions to \eqref{sev1}-\eqref{sev3}, including a study of the index of their horizons. In Section \ref{top} we prove Theorem A. We give some preliminaries in Min-Max theory and present the proof of Theorems B and C in Section \ref{MMC}. The Section \ref{Area-Charge ineq} is devoted to prove the Area-charge inequality and the Theorems D and E. Finally, in the Appendix we discuss the existence and regularity of solutions of a Plateau type problem, necessary to the proof of Theorem B.\\

\noindent
{\bf Acknowledgements}. We would like to thank Lucas Ambrozio for several enlightening discussions about Static systems, and Fernando C. Marques and Andr\'e Neves for their interest in this work. V. Lima thanks Pedro Gaspar and Marco Guaraco for their patience and willingness to explain the details of the Min-max theory, and for their suggestions on the improvement of the text. The authors wish to thank to the organizers of the IX Workshop on Differential Geometry in Macei\'o, where part this research carried out.
T. Cruz and A. de Sousa also wish to express their gratitude to the organizers of the \( 2^{\circ} \) Encontro em Geometria Diferencial no Rio Grande do Sul, when part of this research also was carried out. A. de Sousa would like to thanks the invitation and the hospitality during the visit to the Instituto  de Matem\'atica e Estat\'istica of Universidade Federal do Rio Grande do Sul (UFRGS) in September 2018, when the first ideas of this work came up. 
A. de Sousa was financed by the Coordena\c{c}\~ao de Aperfei\c{c}oamento de Pessoal de N\'ivel Superior - Brasil (CAPES) - Finance Code 001. T. Cruz has been partially suported by  CNPq/Brazil grant 311803/2019-9. Finally, we thank the anonymous referee by its invaluable suggestions and corrections.

\section{ Electrostatic systems}\label{SES}

\subsection{Standard Static Electrovacuum Space-times}\label{App}

 Consider a Lorentzian 4-manifold $(\mathcal{M}^4,\mathfrak{g})$ and a 2-form $F$ on $\mathcal{M}$. The (source-free) Einstein-Maxwell equations with cosmological constant $\Lambda \in \re$ for the triple $(\mathcal{M}, \mathfrak{g}, F)$ are expressed by the following system\footnote{Using a geometric unit system.}
\begin{eqnarray}
  \Ric_{\mathfrak{g}}-\frac{R_{\mathfrak{g}}}{2} \mathfrak{g} + \Lambda \mathfrak{g} = 8\pi T_F,\label{Einstein}\\
  dF=0, \quad \diver_{\mathfrak{g}} F=0 \label{Maxwell}.
\end{eqnarray}
Here $T_F$ denotes the \emph{electromagnetic energy-momentum tensor}
\begin{equation}\label{energ.mom}
T_F = \frac{1}{4\pi}\biggl(F\circ F - \frac{1}{4}|F|_{\mathfrak{g}}^2\,\mathfrak{g}\biggr),
\end{equation}
where $(F\circ F)_{\alpha\beta} = \mathfrak{g}^{\mu\nu}F_{\alpha\mu}F_{\beta\nu}$. The solutions of the \eqref{Einstein}-\eqref{Maxwell} are called \emph{electrovacuum space-times}.

%\begin{rem}Taking the metric trace of \eqref{Einstein} we conclude that the cosmological constant is determined by the scalar curvature of the space-time: $R_{\mathfrak{g}} = 4 \Lambda.$\end{rem}

In what follows, we consider a particular class of electrovacuum space-times. A \emph{standard static space-time} is a pair $(\mathcal{M},\mathfrak{g})$ of the form
\begin{equation*}\label{spacetime}
\mathcal{M} = \re\times M, \ \mathfrak{g} = - V^2 dt^2 + g,
\end{equation*}
where $(M^3,g)$ is an oriented Riemannian $3$-manifold and $V:M\to\mathbb{R}$ is a positive smooth function. Since
\begin{align*}
  \mathcal{L}_{\partial_t} V = 0, \quad \mathcal{L}_{\partial_t} g = 0 \text{ and } |\partial_t|_{\mathfrak{g}}^2 = -V^2 < 0,
\end{align*}
it follows that $\partial_t$ is a time-like Killing field, which induces a time-orientation in the space-time, and each slice $M_t \equiv \{ t \} \times M$ is a space-like hypersurface orthogonal to Killing field, justifying the name static. Moreover, the slices are \emph{totally geodesic} (or \emph{time-symmetric}) and isometric to each other. Assume further  that the space-time admits an \emph{invariant electromagnetic field} with respect to $\partial_t$, i.e.,
\begin{align} \label{eq:invariance--electromagnetic_field}
  \mathcal{L}_{\partial_t} F = 0
\end{align}
has to be satisfied. It is worth highlighting that this Killing field induces an \emph{observer field}, namely, a future-oriented time-like unit field, and $n = V^{-1} \partial_t$ is called the \emph{static observer (field)}. Notice that $n$ coincides with the unit normal field along each slice $M_t$. 

The electromagnetic tensor $F$ admits a unique decomposition in terms of the \emph{electric field} $E$ and the \emph{magnetic field} $B$, as measured by the static observer~\cite[Section 6.4.1]{GourgoulhonFormalismGeneralRelativity2012},
\begin{align}\label{eq:electric_magnetic_field}
  E^{\flat}  = - \iota_n F,\quad
  B^{\flat} = \iota_n \star_{\mathfrak{g}} F, \quad F =   E^{\flat} \wedge dt + V^{-1} \iota_B \iota_{\partial_t} \omega_{\mathfrak{g}},
\end{align}
where $\omega_{\mathfrak{g}}$ and $\star_{\mathfrak{g}}$ denote, respectively, the volume form and the Hodge star operator on differential forms, with respect to the metric $\mathfrak{g}$, and $\iota$ denotes the (left) interior multiplication on tensors. By construction, since $F$ and its Hodge dual $\star_{\mathfrak{g}}F$ are antisymmetric tensors, we have
\[ \langle  {E, n} \rangle_{\mathfrak{g}} = \langle {B, n} \rangle_{\mathfrak{g}} = 0. \]
This means that the electric and magnetic fields are tangent to the slices $M_t$. Therefore, on each slice the metric duality in the equations \eqref{eq:electric_magnetic_field} coincides with that induced by the metric $g$. Furthermore, using the fact that $\partial_t$ is a Killing field we get that  \eqref{eq:invariance--electromagnetic_field} is equivalent to
\begin{align*}
  \mathcal{L}_{\partial_t} E = 0 \text{ and } \mathcal{L}_{\partial_t} B = 0.
\end{align*}

A computation shows that on a standard static electrovacuum space-time
the equations \eqref{Einstein} and \eqref{Maxwell} are equivalent
to the following system on $(M,g)$
\cite[Sections 5.1 and 6.4]{GourgoulhonFormalismGeneralRelativity2012}:
\begin{align*}
  \hess_{g}V &= V\left[ {\Ric_g} - \Lambda g + 2\bigl( E^{\flat}\otimes E^{\flat} + B^{\flat}\otimes B^{\flat} \bigr)
               - \left(|E|^2 + |B|^2 \right)g \right],\\
  R_g &=  2\left(\Lambda + |E|^2 + |B|^2\right),\\
  \iota_B \iota_E\,\omega_{g} &= 0,\\
  \diver_{g}E &= 0, \ \curl_g (VE) = 0, \\
  \diver_{g}B &= 0, \ \curl_g (VB) = 0,
\end{align*}
where $\omega_{g}$ denotes the Riemannian volume form with respect to the metric $g$.
It follows from the third equation equation above that there is a function $\sigma$ such that $B = \sigma E$. Observe that
\begin{align*}
0 = \diver_g (\sigma E) = g(\nabla \sigma,E),\quad 0 = \curl_g (\sigma VE) = V\nabla \sigma\times E.
\end{align*}
Since $V > 0$, we obtain $\nabla \sigma = 0$. Thus $\sigma$ is constant.

Therefore, the system of equations can be rewritten in the simplified form:
\begin{align}
  \hess_{g}V &= V\left[ {\Ric_g} - \Lambda g
               + 2 X^{\flat}\otimes X^{\flat}
               - |X|^2 g\right], \label{stelvc1_simplified}\\
  R_g &=  2\left[ \Lambda + |X|^2 \right],
        \label{stelvc2_simplified}\\
  \diver_{g}X &= 0, \ \curl_g (VX) = 0 \label{stelvc4_simplified},
\end{align}
where $X = \sqrt{1 + \sigma^2}\,E$.

\begin{rem}
  The fact that \( B = \sigma E \) may sound puzzling
  for the reader, since it is well-known that for
  electromagnetic waves solving Maxwell's equations in Minkowski spacetime we have that $E$ and $B$ are orthogonal.
  However, in this case neither $F$ is static, nor the coupled Einstein-Maxwell equations are satisfied (the Maxwell equations hold, however the spacetime is flat).
\end{rem}

Taking the metric trace in \eqref{stelvc1_simplified}, one deduce that \eqref{stelvc2_simplified} is equivalent to
\begin{align}
  \Delta_g V = \left[ (1 + \sigma^2) |E|^2 - \Lambda \right] V. \label{stelvc6}
\end{align}

For simplicity, take \( B \equiv 0 \), equivalently, \( \sigma = 0 \).
So, the previous equations reduce to \eqref{sev1}, \eqref{sev2}
and \eqref{sev3}.

\begin{rem}
Observe that the equations \eqref{stelvc1_simplified}, \eqref{stelvc4_simplified} and \eqref{stelvc6} have the same format that \eqref{sev1}, \eqref{sev2} and \eqref{sev3}. As a consequence, the results proved in this work continue to hold in the case of systems arising from static electrovacuum spacetimes where we have a pure magnetic field, or where both the electric and the magnetic fields are present.
\end{rem}

\subsection{General properties of electrostatic systems} \label{sec:gpes}

%The goal of this  section is to investigate consequences of the existence of a nontrivial solution satisfying the equations \eqref{sev1}, \eqref{sev2} and \eqref{sev3}.

\begin{lemma}\label{lemma:propert}
Consider $(M^3,g)$ be a Riemannian manifold, $E \in \mathfrak{X}(M)$ and $V\in C^{\infty}(M)$ satisfying the following system of equations:
  \begin{align}
  \hess_{g} V & = V(\Ric_{g} - \Lambda g+2E^\flat\otimes E^\flat-|E|^2g), \label{eq1} \\
  \Delta_{g} V & = \left( |E|^2 - \Lambda \right)V, \label{eq2}
 \end{align}
for some constant $\Lambda \in \re$. Suppose $V$ is not identically zero and  $\protect{\Sigma=V^{-1}(0)}$ is nonempty, then:
\renewcommand{\labelenumi}{\roman{enumi})}
\begin{enumerate}
 \item $R_g= 2|E|^2 + 2\Lambda$.
 \item $\Sigma$ is a totally geodesic hypersurface and $|\nabla_g V|$ is a positive constant on each connected component of $\Sigma$;
 \item If $\curl_g (VE) = 0$, then $E$ and $\nabla_g V$ are linearly dependent along $\si$.
\end{enumerate}
\end{lemma}

\begin{proof}  Taking the trace of (\ref{eq1}), we obtain
 \begin{equation*}
  \Delta_{g} V=(R_g-3\Lambda-|E|^2)V,
 \end{equation*}
 which compared with \eqref{eq2} implies $R_g-2|E|^2=2\Lambda.$
 
For item ii), suppose that $\nabla_g V(p) = 0$ for some  $p\in \Sigma$, then along any geodesic $\gamma(t)$ starting from $p,$ the function $V(t)= V(\gamma(t))$ satisfies
 \begin{align*}
    V''(t) &=\hess_{g} V(\gamma'(t),\gamma'(t))=\phi(t)V(t),
 \end{align*}
where $\phi(t) := \left( \Ric_{g} - \Lambda g+2E^\flat\otimes E^\flat-|E|^2g \right)\left( \gamma'(t), \gamma'(t) \right)$. Since $V'(0)=V(0)=0$, we have that $V$ is identically zero near $p$ by uniqueness for solutions to second order ODE's. By using analytic continuation of solutions of elliptic equations (see for instance \cite{Aron}) with respect to the equation \eqref{eq2}, we can conclude that $V$ vanishes identically on $M$. But this contradicts  the fact that $V$ is nontrivial on $M$.  Therefore  $\nabla_g V(p)\neq 0,$ for all $p\in \Sigma,$  which implies that $\Sigma$ is an embedded surface. Moreover, given  any tangent vectors $X, Y$  to $\Sigma,$ we see  that $\hess_{g} V(X,Y)=0$ and $\hess_{g} V(X,\nabla_g V)=0$ on $\Sigma,$ then we obtain that $\Sigma$ is totally geodesic and $|\nabla_g V|^2$ is a constant along $\Sigma$.
Hence the assertion ii) follows.

Finally, observe that $\curl_g(VE)=0$ holds if, only if,
\( d( V E^{\flat} ) = 0 \). So
\begin{equation*}\label{conseqcurl}
    dV \wedge E^\flat + V dE^\flat = 0.
\end{equation*}
If $V=0$, it follows that $dV \wedge E^\flat = 0$. Thus, item iii) holds.
\end{proof}

In the following, \( (DR_g)^{\ast} \) denotes the formal adjoint linearization of the scalar curvature, i.e.,  \((DR_g)^{\ast}(V) = \hess_g V - \Delta_g V g - V \Ric_g\).

\begin{prop}
  Consider \( (M^3, g, V, E) \) as in Lemma \ref{lemma:propert}, and such that  
 $$\diver_g E = 0,\quad \curl_g (VE) = 0.$$ 
 Then, the equations  \eqref{eq1}
  and \eqref{eq2} are equivalent to the following second order overdetermined
  elliptic equation:
  \begin{align}
    \hess_g V - \Delta_g V g - V \Ric_g
    = 2V \left( E^\flat\otimes E^\flat - |E|^2 \ g \right) \label{eq:overdeterm}
  \end{align}
  In particular, \( V \in \mathrm{Ker}(DR_g)^{\ast}\) if, and only if, \( E \equiv 0 \).
\end{prop}
\begin{proof}
  The equations \eqref{eq1} and \eqref{eq2} imply \eqref{eq:overdeterm}
  immediately.

  Let us prove the converse. Taking the metric trace in \eqref{eq:overdeterm}
 we obtain
  \begin{displaymath}
    \Delta_g V = \left( 2 |E|^2 - \frac{R_g}{2} \right) V.
  \end{displaymath}
  On the other hand, taking the divergence of \eqref{eq:overdeterm},
  it follows that
 \begin{align*}
  \diver_g(\hess_g V) =& \Ric_g(\nabla_g V,\cdot)+V\frac{dR_g}{2}+d(\Delta_g V)\\
                       & + 2\langle\nabla_g V,E\rangle E^{\flat}
                         + 2V(\nabla_E E)^{\flat} - 2d(|E|^2V),
 \end{align*}
 where we have used the equation $\diver_g E=0$. Using the identity
 \begin{displaymath}
   \diver_g(\hess_g V)=d(\Delta_g V)+ \Ric_g(\nabla_g V,\cdot),
 \end{displaymath}
 the previous equation can be rewritten in the form,
 \begin{displaymath}
   V d\left( |E|^2 - \frac{R_g}{2} \right) = 2 \left(
   \langle\nabla_g V,E\rangle E^{\flat} + V(\nabla_E E)^{\flat} - |E|^2 dV \right)
   - V d|E|^2.
 \end{displaymath}
 A computation shows that the right side is exactly \( 2 \iota_E d(V E^{\flat}) \).
 Thus,
 \begin{displaymath}
   V d\left( |E|^2 - \frac{R_g}{2} \right) = 2 \iota_E d(V E^{\flat}) = 0,
 \end{displaymath}
 where it was used the fact that the equation $\curl_g (VE) = 0$ holds if, and
 only if, \( d(V E^{\flat}) = 0 \). Therefore, \(  |E|^2 - \frac{R_g}{2} \) is a constant. 
 Putting this together with the identity for laplacian,
 it follows that the equation \eqref{eq:overdeterm} implies 
 \eqref{eq1} and \eqref{eq2}.

 In particular, \( V \in \mbox{Ker}(DR_g)^{\ast}\)  if, and only if, \( 
 V \left( E^\flat\otimes E^\flat - |E|^2 \ g \right) = 0 \). Taking the metric trace,
we conclude that $ E^\flat\otimes E^\flat - |E|^2 \ g = 0$  is equivalent to \( E \equiv 0 \).

It remains to prove that \eqref{eq:overdeterm} is a second order overdetermined elliptic equation. Indeed, consider the operator
  \begin{displaymath}
    Pu = \hess_{g} u - \Delta_{g} u g - u\Ric_{g}
    - 2u \left( E^\flat\otimes E^\flat - 2|E|^2 g \right)
\end{displaymath}
on \( M \), whose principal symbol is
\( \left[\sigma_{\xi}\left(P\right)_p u\right]_{i j}
= \left(-|\xi|^{2}g_{i j}+\xi_{i} \xi_{j}\right)u \), for \(\xi\in T_p^*M \).
This map is injective for $\xi\neq 0.$  Hence, $P$ is overdetermined elliptic.
\end{proof}

\section{Standard Models}\label{ME}

\subsection{List of models}

In this section, we will describe some examples of electrostatic systems. The models are indexed by three parameters: $\Lambda>0$, $m \geq 0$ and $Q \in \re$, which are interpreted respectively as the cosmological constant, the mass and the charge of the associated spacetimes. These constants satisfy (see \cite{BOU})
$$
m^2\leq\frac{1}{18\Lambda}\Big[1+12Q^2\Lambda+(1-4Q^2\Lambda)^{3/2}\Big].
$$
In particular $m^2\Lambda\leq 2/9$ and $Q^2\Lambda \leq 1/4$. In Figure \ref{Rotulo}, we present an illustrative graph relating the models with the parameters $(Q,m,\Lambda).$

Now, we present the list of the models. In each case we describe $(M,g,V,E)$, where $M = \{x;\, V(x) \geq 0\}$ and $g$ is the metric on $\{x;\, V(x) > 0\}$. Additionally we present the values of $m$ and $Q$. In the following, $g_{\mathbb{S}^n}$ denotes the metric on $\mathbb{S}^n$ of constant curvature $1$, while $\rho$ is a constant such that
$$m=\rho\left(1-\frac{2}{3} \Lambda \rho^{2}\right)\quad\mbox{and}\quad   Q^2=\rho^2\left(1-\Lambda \rho^{2}\right).$$

\begin{figure}[!htb]
\centering
\includegraphics[scale=0.21]{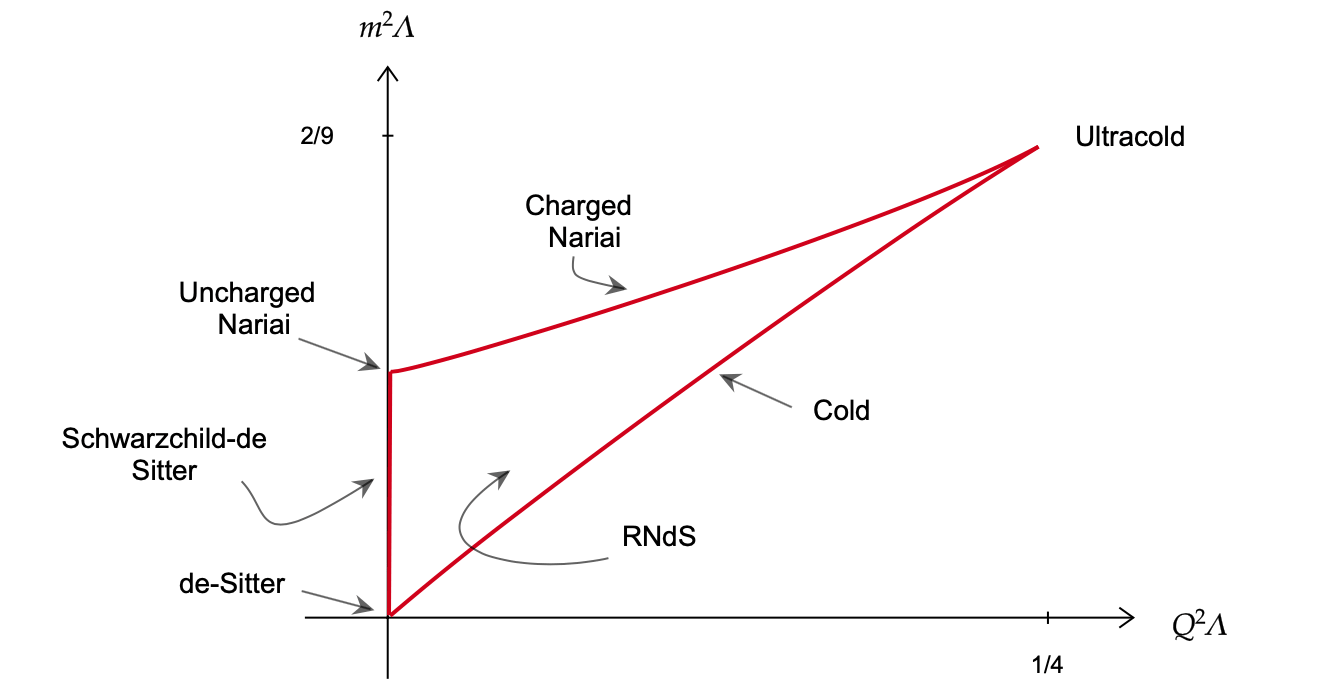}
\caption{
Standard models. The plot is of the dimensionless quantities $m^2\Lambda$ vs. $Q^2\Lambda$.}
\label{Rotulo}
\end{figure}

\begin{itemize}
\item \textbf{de Sitter system}:
\begin{align*}
m = 0, \quad Q=0,\quad M = \mathbb{S}^{3}_{+},\quad g=\frac{3}{\Lambda}\,g_{\mathbb{S}^3},\quad V = x_4, \quad E=0;
\end{align*}

\item \textbf{Schwarzschild-de Sitter system}: 
\begin{align*}
&0 < m < \frac{1}{3\sqrt{\Lambda}}, \quad Q=0,\\
&M = [r_+,r_c]\times\mathbb{S}^{2},\quad g=\left(1-\frac{2 m}{r}-\frac{\Lambda r^2}{3}\right)^{-1}dr^2 + r^2g_{\mathbb{S}^2}\\
&V = \left(1-\frac{2 m}{r}-\frac{\Lambda r^2}{3}\right)^{\frac{1}{2}}, \quad E=0;
\end{align*}

\item \textbf{Nariai system}: 
\begin{align*}
&m = \frac{1}{3\sqrt{\Lambda}}, \quad Q=0,\quad M = [0,\pi/\alpha]\times\mathbb{S}^{2},\quad g=ds^2 + \rho^2g_{\mathbb{S}^2}\\
&V = \sin(\alpha s), \quad E=0, \quad \alpha = \sqrt{\Lambda};
\end{align*}

\item \textbf{Reissner-Nordstr\"om-de Sitter system}: 
\begin{align*}
&0 < m^2 < \frac{1}{18\Lambda}\Big[1+12Q^2\Lambda+(1-4Q^2\Lambda)^{3/2}\Big],\quad 0 < Q^2 \leq \frac{1}{4\Lambda},\\
&M = [r_+,r_c]\times\mathbb{S}^{2},\quad g=\left(1-\frac{2 m}{r}+\frac{Q^2}{r^2}-\frac{\Lambda r^2}{3}\right)^{-1}dr^2 + r^2g_{\mathbb{S}^2}\\
&V = \left(1-\frac{2 m}{r}+\frac{Q^2}{r^2}-\frac{\Lambda r^2}{3}\right)^{\frac{1}{2}}, \quad E=\frac{Q}{r^2}\,V(r)\,\partial_r;
\end{align*}

\item \textbf{Charged Nariai system}:
\begin{align*}
&0 < m^2 = \frac{1}{18\Lambda}\Big[1+12Q^2\Lambda+(1-4Q^2\Lambda)^{3/2}\Big], \quad 0 < |Q| \leq \sqrt{\Lambda}\,\rho^2,\\
&\frac{1}{2\Lambda} < \rho^2 < \frac{1}{\Lambda}, \quad M = [0,\pi/\alpha]\times\mathbb{S}^{2},\quad g=ds^2 + \rho^2g_{\mathbb{S}^2}\\
&V = \sin(\alpha s), \quad E=\frac{Q}{\rho^2}\,\partial_s, \quad \alpha = \sqrt{\Lambda - Q^2/\rho^4};
\end{align*}

\item \textbf{Cold black hole system}:
\begin{align*}
&0 < m^2 = \frac{1}{18\Lambda}\Big[1+12Q^2\Lambda+(1-4Q^2\Lambda)^{3/2}\Big], \quad |Q| \geq \sqrt{\Lambda}\,\rho^2 ,\\
&0 < \rho^2 < \frac{1}{2\Lambda}, \quad M = [0,+\infty)\times\mathbb{S}^{2},\quad g=ds^2 + \rho^2 g_{\mathbb{S}^2}\\
&V = \sinh(\beta s), \quad E=\frac{Q}{\rho^2}\,\partial_s,\quad \beta = \sqrt{Q^2/\rho^4-\Lambda};
\end{align*}

\item \textbf{Ultracold black hole system}:
\begin{align*}
&m = \frac{1}{3}\left(\frac{2}{\Lambda}\right)^{\frac{1}{2}},\quad Q^2 = \frac{1}{4\Lambda} = \rho^2,\quad M = \left[0,+\infty\right)\times\mathbb{S}^{2},\\ 
&g=ds^2 + \rho^2g_{\mathbb{S}^2},\quad V = s, \quad E=\sqrt{\Lambda}\,\partial_s.
\end{align*}
\end{itemize}

In the following subsections we describe how to obtain the examples listed above from known space-times and how they can be extended to complete systems. In addition, we present a quotient of the Nariai system in subsection \ref{quotients}.

\subsection{The Reissner-Nordstr\"om solution}
The Reissner-Nordstr\"om black hole space-time in dimension $3+1$, of mass parameter $m > 0,$ electric charge $Q \in \re$ and cosmological constant $\Lambda \in \re,$ is described by the metric 
\begin{equation}\label{model}
\mathfrak{g}=-\left(1-\frac{2m}{r}+\frac{Q^2}{r^2}-\frac{\Lambda r^2}{3}\right)dt^2+\frac{dr^2}{\left(1-\frac{2m}{r}+\frac{Q^2}{r^2}-\frac{\Lambda r^2}{3}\right)}+r^2g_{\sw},
\end{equation}
 where $(\sw, g_{\sw})$ denotes the  round sphere, $r$ is a radial parameter varying in a suitable open set $I\subset(0,\infty)$, and $t \in \re.$
These metrics correspond to spherically symmetric Static Electrovacuum space-times (as defined in Subsection \ref{App}), with potential and electromagnetic tensor given respectively by
$$V(r) = \biggl(1-\frac{2m}{r}+\frac{Q^2}{r^2}-\frac{\Lambda r^2}{3}\biggr)^{\frac{1}{2}}, \ F = \frac{Q}{r^2}\,dr\wedge dt.$$

From now on, assume $\Lambda>0$. The solution given by \eqref{model} is the so called  {\it  Reissner-Nordstr\"om-de Sitter (RNdS) space-time}. In this setting, the set $I$ depends on the solution of the following equation: 
\begin{equation}\label{lapse}
1-\frac{2m}{r}+\frac{Q^2}{r^2}-\frac{\Lambda r^2}{3}=0 \ \Leftrightarrow \ \frac{\Lambda}{3}r^4 - r^2 + 2mr - Q^2 = 0.
\end{equation}
This polynomial equation has four solutions, which from now on we assume are all real and distinct. Using Vieta's formulas we conclude that one of the roots is negative. We denote by $r_c > r_+ > r_-$ the positive roots of \eqref{lapse}. The physical significance of these numbers is that $\{r = r_-\}$ is the inner (Cauchy) black hole horizon, $\{r = r_+\}$ is the outer (Killing) black hole horizon and $\{r = r_c\}$ is the cosmological horizon, and the smallest root  has no physical significance, since is negative.

Now, we want to consider the slice $\{t = 0\}$. Observe that $1-\frac{2m}{r}+\frac{Q^2}{r^2}-\frac{\Lambda r^2}{3}$ is positive in the region where $r \in (r_+,r_c)$, and it is negative in the region where $r \in (r_-,r_+)\cup(r_c,+\infty)$. So, in order to obtain a Riemannian manifold associated to an electrostatic system, we consider $(r_+,r_c)\times \sw$, endowed with the metric 
\begin{equation}\label{rnds}
g_{m,Q,\Lambda} = \left(1-\frac{2m}{r}+\frac{Q^2}{r^2} -\frac{\Lambda r^2}{3}\right)^{-1}dr^2+r^2g_{\sw}.
\end{equation}
We point out that this metric is not defined when $r\in\{r_+,r_c\},$ but these are just coordinate singularities and \eqref{rnds} can be  extended smoothly to $[r_+,r_c]$, as we shall see.

The electric field is given by $E = \frac{Q}{r^2}V(r)\,\partial_r$. In a spherical slice, $S_r=\{r\}\times \sw$, where $r \in (r_+,r_c),$ a unit normal is
$N_r = V(r)\,\partial_r$. So it is easy to see that for all $r \in (r_+,r_c)$, the charge of the slice $S_r$ with respect to $N$ (as defined in the introduction and in Section \ref{Area-Charge ineq}) is equal to $Q$.

In order to extend the metric \eqref{rnds} to include the horizon boundaries, we have to perform the following change of variables 
$$
s(r) :=\int_{r_{+}}^{r}\left(1-\frac{2 m}{t}+\frac{Q^{2}}{t^{2}}-\frac{\Lambda t^2}{3}\right)^{-1 / 2} d t.
$$

We have $\frac{ds}{dr} = \frac{1}{V(r)} > 0$, for $r \in (r_+,r_c)$, so by the Inverse Function Theorem there is $u:(0, a) \rightarrow\left(r_{+}, r_{c}\right)$ which is the inverse of $s(r)$. By continuity, $u$ extends to $[0, a]$ with $u(0) = r_{+}$, $u(a) = r_{c}.$
This allow us to rewrite the metric \eqref{rnds} as, 
\begin{equation*}\label{ERN}
g_{m,Q,\Lambda}=d s^{2}+u(s)^{2}g_{\sw},  \ \textrm{on} \  [0,a] \times \sw.
\end{equation*}

Now we want to extend the metric \eqref{ERN} to a complete smooth metric on $\re \times \sw$. The idea is to define
\begin{equation*}\label{FMRN}
g = d s^{2}+v(s)^{2}g_{\sw},  \ \textrm{on} \  \re \times \sw,
\end{equation*}
where $v: \re \to (0,+\infty)$ is a smooth periodic function such that $v|_{[0,a]} \equiv u$.

First define $v$ in $[0,2a]$ as
\begin{equation*}
v(s)=
 \left \{
\begin{array}{cc}
u(s),& s\in [0,a]\\ 
\tilde u(s),& s\in [a,2a]\\
\end{array}
\right.,
\end{equation*}
where $\tilde{u}(s)=u(2a-s).$
We have to prove that the right and left derivatives of order $k$ of $v$ at $a$ coincide, for all $k \in \mathbb{N}$.

One can check that 
\begin{eqnarray*}
u^{\prime}(s) &=& \left(1-\frac{2 m}{u(s)}+\frac{Q^{2}}{u(s)^{2}}-\frac{\Lambda u(s)^2}{3}\right)^{1 / 2} = V\left(u(s)\right),\label{firstd}\\
u^{\prime \prime}(s) &=& \left(\frac{ m}{u^2(s)}-\frac{Q^{2}}{u(s)^{3}}-\frac{\Lambda u(s)}{3}\right)=:G(u(s)),\label{secondderi}
\end{eqnarray*}
where $G$ is a real valued function. For the others derivatives, we can use the Fa\`a di Bruno's formula
\begin{equation*}\label{faa}
\frac{d^{k}}{d s^{k}} G(u(s))=\sum \frac{k !}{m_{1} ! m_{2} ! \cdots m_{k} !} \cdot G^{(m)}(u(s)) \cdot \prod_{j=1}^{k}\left(\frac{u^{(j)}(s)}{j !}\right)^{m_{j}},
\end{equation*}
where $m=m_{1}+\cdots+m_{k}$, and the sum is taken over all partitions of $k$, that is, collections of non negative integers $m_1, \ldots, m_k$ satisfying the constraint 
\begin{equation*}\label{constraint}
m_1+2m_2+\cdots +k\cdot m_k=k.
\end{equation*}

By the chain rule, 
 $\frac{d^k}{ds^k}\tilde u(a^-)=(-1)^k\frac{d^k}{ds^k}u(a^+).$ Thus $\frac{d^k}{ds^k}\tilde u(a^-)= \frac{d^k}{ds^k}u(a^+)$ for $k$ even.
We will prove by induction that
\begin{equation*}\label{indhp}
 \frac{d^k}{ds^k}u(a^+)=0=\frac{d^k}{ds^k}\tilde u(a^-),
\quad \forall \ \mbox{k $\in\{2p-1,\;p\in\mathbb{N}\}.$ }   
\end{equation*}
For $k=1$, using \eqref{firstd} we obtain $u'(a^+)=V(r_c)=0=-V(r_c)=\tilde u'(a^-)$.  For $k=3,$ we see that $$u'''(a^+)=G'(u(a^+))\cdot u'(a^+)=0=G'(\tilde u(a^-))\cdot \tilde u'(a^-)=\tilde u'''(a^-).$$
Suppose that \eqref{indhp} holds for $k=1,3,\cdots,2p-1.$ By \eqref{faa} we have
\begin{eqnarray*}
    \frac{d^{2p+1}}{ds^{2p+1}}u(a^+)&=&\frac{d^{2p-1}}{ds^{2p-1}}[G(u(a^+))]\\
    &=& G^{(2p-1)}(u(a^+))u^{(2p-1)}(a^+)+O(a^+),
\end{eqnarray*}
where $O(a^+)$ is a sum of terms each of which is a product of factors, including a right hand derivative of $u$ at $a$ of order $k \in \{1,3,\cdots,2p-1\}$ (this holds because $m_1+2m_2+\cdots +(2p-1)m_{2p-1}=2p-1$ implies that at least one of the $m_i$'s is odd and less than or equal to $2p-1$). So, by the induction assumption, $\displaystyle\frac{d^{2p+1}}{ds^{2p+1}}u(a^+)=0$. An analogous computation give us $\displaystyle\frac{d^{2p+1}}{ds^{2p+1}} \tilde u(a^-)=0$. Thus, \eqref{indhp} holds. 

Finally, extend $v$ periodically for all $s \in \re$ putting $v(s + 2a) = v(s), \forall \ s \in \re$. Consider $s_n = 2na, n \in \mathbb{Z}$. Reasoning as before one can prove that the right and left derivatives of order $k$ of $v$ at $s_n$ coincide, for all $k \in \mathbb{N}$. Therefore $v$ is a smooth function.

Define $M_n = [na,(n+1)a]\times\sw$, $n \in \mathbb{Z}$. It follows from the definitions of $v$ and $g$ that $\left(M_n,g\right)$ is isometric to $\left(M_0,g\right)$, $\forall \ n \in \mathbb{Z}$, and this last submanifold corresponds to the original model. So, defining 
$$V(s,x) = \left(1-\frac{2 m}{v(s)}+\frac{Q^{2}}{v(s)^{2}}-\frac{\Lambda v(s)^2}{3}\right)^{\frac{1}{2}}, \ E(s,x) = \frac{Q}{v(s)^2}\,\partial_s,$$
$(\re\times\sw,g,V,E)$ is a complete electrostatic System, which we call the de-Sitter Reissner-Nordstr\"om system. Considering $Q = 0$, we obtain the Schwarschild-de Sitter system.

\subsection{The de Sitter solution}

The de Sitter solution is given by \eqref{model}, with $Q=m=0$. In particular $F \equiv 0$ and we are in the vacuum case. Performing the coordinate transformation $r=\sqrt{\frac{3}{\Lambda}}\cos s$, we obtain
\begin{equation}\label{deSitter}
\mathfrak{g} = - \sin^{2}s\,dt^{2} + \frac{3}{\Lambda}\left(ds^{2} + \cos^2s\,g_{\sw}\right)
\end{equation}
The metric on the slices $\{t = \textrm{constant}\}$ is well defined for $s \in (-\pi/2,\pi/2)$, and it is the metric of a round three-sphere of radius $(\Lambda/3)^{-1/2}$ in the rotationally symmetric form. This expression excludes two points of the manifold, however with a further change of coordinates we cover these points and obtain $g=\frac{3}{\Lambda}\,g_{\mathbb{S}^3}$, $V = x_4,$ where $x_4$ is the standard coordinate on $\mathbb{R}^4$.

\subsection{The (charged) Nariai solution}\label{Nariai}

We will now describe a model having a certain double root of the potential $V$.

First of all, suppose that $r_1$ and $r_2$, with $r_1>r_2$, are positive zeros of $V(r)=0.$  Then, the equations $V(r_1)=0$ and  $V(r_2)=0$ can be  rearranged to 
\begin{equation*}
2m=(r_1+r_2)\left(1-\frac{r_1^{2}+r_2^{2}}{3} \Lambda\right)\quad\mbox{and}\quad   Q^{2}=r_1r_2\left(1-\frac{r_1^{2}+r_1r_2+r_2^{2}}{3} \Lambda\right).    
\end{equation*}

Assume that there is a double root $\rho$ of $V(r)$. We have
\begin{eqnarray}
 V(r)^2 = \frac{(r-\rho)^{2}}{r^2}\left(1-\frac{1}{3} \Lambda\left(r^{2}+2 \rho r+3 \rho^{2}\right)\right),\label{newpot}\\
 m=\rho\left(1-\frac{2}{3} \Lambda \rho^{2}\right)\quad\mbox{and}\quad   Q^2=\rho^2\left(1-\Lambda \rho^{2}\right). \label{masscharge}
\end{eqnarray}
  
Since the  cosmological constant $\Lambda$ is positive, we have that $\rho\in \left(0,\frac{1}{\sqrt{\Lambda}}\right).$ Moreover, the other positive root of $V$ is equal to
$$\rho_*=\left( \frac{3}{\Lambda}-2\rho^2\right)^{1/2}-\rho.$$

The {\it charged Nariai spacetime} is a special solution obtained if  $\frac{1}{2\Lambda}<\rho^2<\frac{1}{\Lambda}$. We point out that this  is equivalent to $r_+=r_c=\rho.$ Now, if  $\rho^2\in\left(0,\frac{1}{2\Lambda}\right)$, which  is equivalent to $r_+=r_{-}=\rho,$ this model is called of  {\it cold black hole}. The case where  $r_-=r_+=r_c=\frac{1}{\sqrt{2\Lambda}}$ is called of {\it ultra cold black hole}. See \cite{Rom} for the reasons behind this nomenclature.

In the case of the charged Nariai the potential $V$ degenerates, since $r_+ = r_c$. So, it is necessary to make a suitable coordinate transformation and rewrite the metric \eqref{model}. Let $\rho + \epsilon > \rho - \epsilon > r_{-} > 0 > r_{*}$ be the roots of $V$. Then 
$$V(r)^2 =-\frac{\Lambda}{3 r^{2}}(r - \rho -\epsilon)(r - \rho +\epsilon)\left(r-r_{-}\right)\left(r-r_{*}\right).$$

If one perform the following  coordinate transformation in \eqref{model}: 
\begin{equation*}
r= \rho + \epsilon \cos (\alpha s)
\quad\mbox{and}\quad \tau = \alpha\epsilon t,
\end{equation*}
where $\alpha = \displaystyle\sqrt{\frac{\Lambda\left(\rho - r_{-}\right)\left(\rho - r_{*}\right)}{3\rho^2}}$, one obtains letting $\epsilon\to 0$ that
\begin{eqnarray}\label{nar-cilinder}
\mathfrak{g} &=& -\sin^{2}(\alpha s)\,d\tau^{2} + ds^{2} + \rho^2 g_{\sw},\\
F &=& \frac{Q\sin (\alpha s)}{\rho^2} ds\wedge d\tau,
\end{eqnarray}
where $\tau \in \re$ and $s \in (0,\pi)$. Since $r_{+}$ and $r_{-}$ are simple roots of $V$ and $\rho$ is a double root, by using Vieta's formulas in \eqref{lapse}, we see that $r_{*} + r_{-} = -2\rho$ and $r_{*}r_{-} = -\frac{3Q^2}{\Lambda\rho^2}$, thus $\alpha = \sqrt{\Lambda - Q^2/\rho^4}$.

The horizons are now located in $s=0$ and $s = \pi/\alpha$. Each slice $\{\tau = \ \textrm{constant}\}$ is a cylinder with a standard product metric, which extend smoothly for all $s \in \re$. So, $(\re\times\sw,g,V,E)$ is an electrostatic system, where $g = ds^{2} + \rho^2 g_{\sw}$, $V(s) = \sin(\alpha s)$ and $E(s) = \frac{Q}{\rho^2}\,\partial_s$. In the special case $Q = 0$, we have the Nariai system.

\subsection{A quotient of the Nariai system}\label{quotients}

As pointed out in \cite[Section 7]{A} there is an interesting quotient of the Nariai system. Consider the map $\mathcal{I}: [0,\pi/\alpha]\times\mathbb{S}^{2} \to [0,\pi/\alpha]\times\mathbb{S}^{2}$ defined by $\mathcal{I}(s,x)=(-s+\pi/\alpha,-x)$, and the group $\Gamma$ generated by $\mathcal{I}$, which is isomorphic to $\mathbb{Z}_2$. Observe that $M = \big([0,\pi/\alpha]\times\mathbb{S}^{2}\big)/\Gamma$ is homeomorphic to $\mathbb{RP}^3$ minus an open ball. The metric $g$ induces a metric $\widetilde{g}$ in the quotient. Moreover $V$ satisfies $V\big(\mathcal{I}(x)\big) = V(x)$, so $\widetilde{V}([x]) = V(x)$ is well defined and satisfies the static equations in $\big(([0,\pi/\alpha]\times\mathbb{S}^{2})/\Gamma,\widetilde{g}\big)$. The boundary of $M$ is a $2$-sphere and coincides with the set $\widetilde{V}^{-1}(0)$.
A similar example in the case $Q \neq 0$ can not be obtained since the electric field is not invariant by $\mathcal{I}$.

%Moreover, the image of the two horizons $\{s=0\}$ and $\{s=\pi/\alpha\}$ via the quotient map is a $\mathbb{RP}^2$, in particular $\widetilde{V}^{-1}(0)$ is connected and non-separating. 

One can check that the double of 
$\big(([0,\pi/\alpha]\times\mathbb{S}^{2})/\Gamma,\widetilde{g},\widetilde{V}\big)$ is smooth, and so it is a static system where  the manifold is homeomorphic to $\mathbb{RP}^3\#\mathbb{RP}^3$. Alternatively, this model can be obtained as a quotient of the Nariai system, defining $\Gamma$ as the group generated by $\mathcal{I}$ and the map $\mathcal{R}:\mathbb{R}\times\mathbb{S}^{2} \to \mathbb{R}\times\mathbb{S}^{2}$ defined by $\mathcal{R}(s,x)=(-s-\pi/\alpha,-x)$.

\subsection{The cold black hole solution}

In the case of the cold black hole solution, reasoning as in the previous case, one perform the following coordinate transformation in \eqref{model}: 
\begin{equation*}
r= \rho + \epsilon \cosh (\beta s)
\quad\mbox{and}\quad \tau = \beta\epsilon t,
\end{equation*}
where $\beta = \displaystyle\sqrt{Q^2/\rho^4 - \Lambda}$. Letting $\epsilon\to 0$ one obtains
\begin{eqnarray}\label{cold}
\mathfrak{g} &=& -\sinh^{2}(\beta s)\,d\tau^{2} + ds^{2} + \rho^2 g_{\sw},\\
F &=& \frac{Q\sinh (\beta s)}{\rho^2}\,ds\wedge d\tau.
\end{eqnarray}
We conclude that $(\re\times\sw,g,V,E)$ is an electrostatic system, where $g = ds^{2} + \rho^2 g_{\sw}$, $V(s) = \sinh(\beta s)$ and $E(s) = \frac{Q}{\rho^2}\,\partial_s$, with a horizon located at $s=0$.

\subsection{The ultra cold black hole solution}

In the case of the ultra cold black hole solution, one perform the following coordinate transformation in \eqref{model}: 
\begin{equation*}
r= \frac{1}{\sqrt{2\Lambda}} + \epsilon \cos \left(\sqrt{\frac{4\epsilon(2\Lambda)^{3/2}}{3}} s\right)
\quad\mbox{and}\quad \tau = \frac{4(2\Lambda)^{3/2}}{3}\epsilon^2 t.
\end{equation*}
Letting $\epsilon\to 0$ one obtains
\begin{eqnarray}\label{ultracold}
\mathfrak{g} &=& -s^2d\tau^{2} + ds^{2} + \frac{1}{2\Lambda} \,g_{\sw},\\
F &=& \sqrt{\Lambda}\,s\,ds\wedge d\tau.
\end{eqnarray}
Thus, $(\re\times\sw,g,V,E)$ is an  electrostatic system, where $g = ds^{2} + \frac{1}{2\Lambda} g_{\sw}$, $V(s) = s$ and $E(s) = \sqrt{\Lambda}\,\partial_s$, with a horizon located at $s=0$.

\subsection{The index of the horizons in the models}\label{sec:index}

Consider a closed minimal surface $\si\subset (M,g).$  If $\si$ is two-sided,  that is, there exists a globally defined unit normal vector field $N$ on $\si$, then the second variation formula with respect to $X=fN$ is given by:
\begin{equation}\label{index}
    \frac{d^{2}}{d t^{2}}\Big|_{t=0} |\Sigma_{t}| =\int_{\Sigma}\big[|\nabla_{\si} f|^{2}-\left(\Ric_g(N, N)+|A|^{2}\right) f^{2}\big]\,d\mu  =: B(f)
\end{equation}
where $\nabla_{\si}$ is the gradient operator on $\si$ and $A$ denotes the second fundamental form of $\si.$  
The quadratic form $B$ is naturally associated to a second-order differential operator, the {\it Jacobi operator} of $\si$
\begin{equation*}\label{jacobi}
L_{\si}=\Delta_{\si}+ \Ric_g(N,N)+|A|^2.
\end{equation*}
The Morse index of $\si$ is defined as the number of negative eigenvalues of $L_{\si}$. If one of the eingenvalues is equal to zero, we say $\si$ is \emph{degenerate}.

We say that $\si$ is stable if $B(f) \geq 0$, $\forall \ f \in C^{\infty}(\si)\setminus \{0\}$. This is equivalently to $\mathrm{index}(\si) = 0$, and to $\lambda_1 \geq 0$, where $\lambda_1$ is the first eigenvalue of $L_{\si}$. If $\lambda_1 > 0$, then $\si$ is said to be \emph{strictly stable}. Finally, $\si$ is called \emph{unstable} if it is not stable.

The Nariai, cold and ultra cold systems are standard cylinders of the form $(\re\times \sw,ds^{2}+\rho^2 g_{\sw}),$ where $\rho$ is a constant. Its horizons are slices $\{s\}\times\mathbb{S}^2$, so are homologically nontrivial. The cold and ultracold models have only one horizon, which is area minimizing, and therefore stable. On the other hand, the extended charged Nariai model has multiple stable horizons, each area minimizing. Moreover in these three cases, the horizons are degenerate.

Consider the Reissner-Nordstr\"om-de Sitter manifold, since the Gauss equation implies 
 $$\Ric_g(N,N)=\frac{R_g}{2}-K_{\Sigma}-\frac{|A|^2}{2},$$
 the Jacobi operator  with respect to $\si_{a}=\{r=a\}$ 
becomes 
\begin{eqnarray*}
L_{\si_{a}}= \Delta_{\si_{a}}+\Lambda+\frac{Q^2}{a^4}-\frac{1}{a^2}.
\end{eqnarray*}  
Suppose $0\leq Q^2< \frac{1}{4\Lambda}.$ If either
\begin{equation*}\label{inst1}
a^2>\frac{1+\sqrt{1-4\Lambda Q^2}}{2\Lambda}
\end{equation*}
or
\begin{equation*}\label{inst2}
0<a^2<\frac{1-\sqrt{1-4\Lambda Q^2}}{2\Lambda},
\end{equation*}
then $\si_{a}$ is unstable. Indeed, letting $f=1$  in \eqref{index}, we have that  $B(1)<0.$ On the other hand, if  
  
\begin{equation*}\label{stab}
    \frac{1-\sqrt{1-4\Lambda Q^2}}{2\Lambda}<a^2< \frac{1+\sqrt{1-4\Lambda Q^2}}{2\Lambda},
\end{equation*}
then $\si$ is strictly stable. 

Recall that  $r_-,r_+$ and $r_c$ are positive and distinct roots of \eqref{lapse}, with 
\begin{equation}\label{roots}
    r_- < \rho_{**} < r_+ < \rho_* < r_c,
\end{equation}
where $\rho_*$ and $\rho_{**}$ are critical roots, corresponding to the cases of double roots of the potential $V$, $r_-=\rho_{**}=r_+$ and $r_+=\rho_*=r_c$. 
By equation \eqref{masscharge} we have   
$$
 \rho_{**}^2=\frac{1-\sqrt{1-4\Lambda Q^2}}{2\Lambda}\quad\mbox{and}\quad  \rho_*^2=\frac{1+\sqrt{1-4\Lambda Q^2}}{2\Lambda}.
$$

By \eqref{roots},  we can conclude that 
\begin{eqnarray*}
r_c^2>\frac{1+\sqrt{1-4\Lambda Q^2}}{2\Lambda},
\end{eqnarray*}

\begin{eqnarray*}
\frac{1-\sqrt{1-4\Lambda Q^2}}{2\Lambda}<r_+^2< \frac{1+\sqrt{1-4\Lambda Q^2}}{2\Lambda}.
\end{eqnarray*}
Hence $\{r=r_c\}$ is unstable and $\{r=r_+\}$ is strictly stable.

Finally, in the de-Sitter model the Riemannian manifold is the round $3$-sphere and $V^{-1}(0)$ is an equator. In this case, it is well-known that the totally geodesic spheres have index one.

\section{Topology of electrostatic systems}\label{top}

\subsection{An area-decreasing flow}
We start by introducing the following definition, which concerns a certain type of surface which appear in the proof of Theorem \ref{char1}. 

\begin{definition}\label{def:alm.prop.emb}
Let $(M^3,g)$ be a Riemannian $3$-manifold and let $\Sigma$ be an immersed surface in $(M^3,g)$. We say that $\si$ is almost properly embedded if there exists subsets $\{p_1,\cdots,p_m\}, \{q_1,\cdots,q_n\}$ of $\partial\si$ such that:
\begin{enumerate}
\item $\{p_1,\cdots,p_m\}\cap\{q_1,\cdots,q_n\} = \emptyset$,
\item $\si\setminus\{p_1,\cdots,p_m\}$ is embedded,
\item $\si\setminus\{q_1,\cdots,q_n\}$ is smooth,
\item $\si\cap\partial M = \partial\si$.
\end{enumerate}
\end{definition}

Let $(M^3,g,V,E)$ be a compact electrostatic system such that $V^{-1}(0) = \partial M$. Consider $\Sigma_0^2$ a compact almost properly embedded orientable minimal surface in $(M^3,g)$ and choose $N_0$ a smooth unit normal vector field on $\Sigma_0$. For some $\delta>0$ sufficiently small consider the following smooth flow $\Phi:[0,\delta)\times \Sigma_0 \rightarrow M$ of $\Sigma_0$  satisfying
\begin{align}\label{flow}
\frac{d}{dt}\Phi_{t}(p) &= V(\Phi_t(p))N_{t}(p),   \quad p\in \Sigma_0, t\in[0,\delta),\\
\Phi(0,p) &= \phi,   \quad p\in \Sigma_0,\nonumber
\end{align}
where $N_t(p):= N(p,t)$ is a unit normal at $\Phi_t(p).$
Since $V$ is a  smooth function on $M$ with $V^{-1}(0)=\partial M$ and $dV\neq 0$ on $\partial M$ (by item ii) of Lemma \ref{lemma:propert}), the metric $g$ on $M\setminus\partial M$ is  conformally compact with defining function $V$, that is,  $\overline{g}=V^{-2}g$  extends  as a smooth metric on $M.$  We also remark that  $(M\setminus\partial M,V^{-2}g)$ has bounded geometry (see \cite{AmmLauNis}), i.e., the Riemann curvature tensor and all of its covariant derivatives are uniformly bounded (The bound depending on the order of the derivative) and the injectivity radius is bounded below. 

A physical motivation to consider $(M,V^{-2}g)$ is the following. Let $\gamma$ be the trajectory of a light ray in a standard static spacetime $(\mathbb{R}\times M,-V^2dt^2 + g)$. Then its projection on $M$ is a geodesic in the metric $V^{-2}g$, \cite[Chapter 8]{Fran}.

One can check that  \eqref{flow} is the flow of $\Sigma$ in $(M\setminus\partial M,V^{-2}g)$ by equidistant surfaces, which is smooth if $\delta$ is less than the injectivity radius of $(M\setminus\partial M,V^{-2}g)$.
Since $V$ vanishes on $\partial M$, we observe that each $\Phi_{t}(\Sigma_0):=\Sigma_t$, $t\in[0,\delta)$, is a compact almost properly embedded surface in $(M,g)$ with the same boundary as $\Sigma_0.$ 

In the next proposition we prove that if we start at a minimal surface, the flow given by (\ref{flow}) does not increase the area. If furthermore we assume that the area is contast along the flow, a splitting result is obtained (compare with \cite[Proposition 14]{A} and \cite[Lemma 4]{G}).

\begin{prop}\label{goodlemma}
 Let $(M^3,g,V,E)$ be a compact electrostatic system, such that $V^{-1}(0) = \partial M$. Suppose $\Sigma_0$ is a connected and compact almost properly embedded orientable minimal surface in $(M,g)$ and $\Phi_t$, $t\in[0,\delta)$, is the flow starting at $\Sigma_0$ defined by (\ref{flow}). Then the function $
 \mathfrak{a}(t)=|\Sigma_t|,$ $t\in[0,\delta)$, is monotone non-increasing.

Assuming that $\mathfrak{a}$ is constant we get that $|E|$ is constant on $\Phi\bigl([0,\delta)\times \Sigma_0\bigr)$ and $V_t := V_{\si_t}$ is constant for each $t \in [0,\delta)$. Moreover:
 \renewcommand{\labelenumi}{\roman{enumi})}
 \begin{enumerate}
 \item If $\Sigma_0$ is closed, then
 \begin{equation*}
  \Phi : \bigl([0,\delta)\times \Sigma_0,(V\circ\Phi)^2dt^2+g_{\Sigma_0}\bigr) \rightarrow (M,g)
 \end{equation*}
 \noindent is an isometry onto its image $U\subset M$, so that $(U, g|_U)$ is isometric to 
\begin{align*}
&\left([0, s^{\ast})\times \Sigma_0, ds^2 + g_{\Sigma_0} \right),\ \textrm{if}\ \big|E| = 0,\\
&\left( [0, s^{\ast}) \times \Sigma_0, \frac{1}{|E|}\,ds^2 + g_{\Sigma_0} \right),\ \textrm{if}\ |E|\ \textrm{is positive},
\end{align*} 
where $(\Sigma_0, g_{\Sigma_0})$ has constant Gaussian curvature $K_{\si_0} = c V^{-3}_{0} + 3|E|^2 + \frac{\Lambda}{3}$, for some $c \in \mathbb{R}$.
\vspace{0.2cm}

 \item If $\partial \Sigma_0$ is not empty, then $\Lambda > 0$, $E|_U \equiv 0$ and
 \begin{equation*}
  \Phi : \bigl([0,\delta)\times (\Sigma_0\setminus \partial\Sigma_0),(V\circ\Phi)^2dt^2+g_{\Sigma_0}\bigr) \rightarrow (M,g)
 \end{equation*}
 \noindent is an isometry onto its image $U\subset M\setminus \partial M$  and $V|_{0}$ is a static potential, where $(\Sigma_0,g_{\Sigma_0})$ is isometric to a domain bounded by geodesics in the sphere $\sw$ of constant Gaussian curvature $K=\frac{\Lambda}{3}$. So, $g|_U$ is isometric to the constant sectional curvature metric $V^2_{0}d\theta^2+g_{\Sigma_0}$.
 \end{enumerate}
\end{prop}

\begin{proof}
Since the equations \eqref{sev1}, \eqref{sev2} and \eqref{sev3} still hold if we replace $V$ by $-V$, without loss of generality we can suppose $V > 0$ on $\interior M$.

Let $\{\Sigma_t\}_{t\geq 0}$ be an outward normal flow of surfaces  with initial condition $\Sigma_0$ and speed $V.$
Let $\vec{H}_t = - H_t N_t,$ where $H_t$ denotes the mean curvature  of $\Sigma_t$. 
Recall the following well known evolution of $H_t$ for deformations of surfaces (see for instance \cite{HP}),
$$\frac{\partial}{\partial t}H_{t} = -L_{\Sigma_t} V,$$
where 
$L_{\Sigma_t}$ is the Jacobi operator of $\si_t$ defined as in \eqref{jacobi}.
Recall that the Laplacians $\Delta_g$ on $M$ and $\Delta_{\Sigma_t}$ on $\Sigma_t$ are related by
$$\Delta_g V = \Delta_{\Sigma_t} V + g(\nabla V, \vec{H}_t) + \hess_g V(N_t,N_t),$$
which together with  \eqref{eq:overdeterm} imply that
$$L_{\Sigma_t} V = g(\nabla V, N_t)H_t + |A_{\si_t}|^2V + 2V \left( |E|^2 - \langle E,N_t\rangle^2 \right).$$
Using Cauchy-Schwarz inequality, we have
\begin{equation*}
 \frac{\partial}{\partial t}H_{t} \leq - g(\nabla V, N_t)H_t.
\end{equation*}
Since $V\geq 0$ and $H_0 = 0$, it follows by Gr\"onwall's inequality that each $\Sigma_t$ must have $H_{t} \leq 0$ for every $t\in[0,\delta)$.
By the first variation of area formula (where $\partial \Sigma_t=\partial \Sigma_0$) $\Sigma_t$ must have area less than or equal to that of $\Sigma_0$.

In the remaining of the proof, assume that the function $\mathfrak{a}$ is constant on $[0,\delta)$. We see by Cauchy-Schwarz inequality that $E$ and $N$ are linearly dependent, so there exists a function $f:U\to\mathbb{R}$ such that $E=fN$. Also, we conclude that $\Sigma_t$ is totally geodesic and, hence, each $(\Sigma_t,g_{\Sigma_{t}})$ is isometric to $(\Sigma_0,g_{\Sigma_0})$. 
 
We may assume that in a neighborhood of $M,$ say $U\approx (-\varepsilon,\varepsilon)\times \Sigma\setminus \partial\Sigma \subset M\setminus \partial M,$ the metric can be written as
\begin{equation}\label{metric}
g = (V\circ\Phi)^2(t,x)\,dt^2 + g_{\Sigma_0}.
\end{equation}

Since $\diver_{g}E=0$ and $\diver_{g}N = H_t = 0$, we have $\frac{\partial f}{\partial t}=0$, which implies that $f$ does not depend of $t$.  
Also, since $\curl (VE)=0$ and $N=V\nabla t$, we have $\nabla(fV^2)\times \nabla t = 0$ and so $fV^2$ is independent of $x$.
Hence, there exists a smooth function $\alpha:(-\varepsilon,\varepsilon)\to\mathbb{R}$ such that 
\begin{equation}\label{eq:trick_galloway}
    f(x)V^2(t,x)=\alpha(t),
\end{equation}
for all $(t,x)\in U$. We can suppose $f\geq 0$, so $f = |E|$ and $|E|V^2$ is constant on each $\Sigma_t$.

Under the variation \eqref{flow}, the second fundamental form evolves according to the equation (see \cite{HP})
\begin{equation*}
  \frac{\partial}{\partial t}A_{ij} =-(\hess_{\Sigma} V_t)_{ij} +V_t(- R_{NiNj} + A_{ik}A_j^k),
 \end{equation*} 
where $V_t = V|_{\si_t}$. Since each \( \Sigma_t \) is totally geodesic, we get
\begin{equation*}
 \hess_{\Sigma_t} V_t (W,Z) = - \langle R(W,N_t)N_t,Z\rangle_g V_t, \quad \forall \, W,Z \in \mathfrak{X}(\Sigma_t).
\end{equation*}
Given an orthonormal basis $\{X,Y,N_t\}$ in $\Sigma_t$, we also get
\begin{align*}
    \Ric_g(X,X) = K_t + \langle R(N_t,X,X),N_t \rangle_g,
\end{align*}
where $K_t$ is the Gaussian curvature of $\si_t$. Combining the last two equations with \eqref{eq1}, we obtain
\begin{align*}
  \hess V_t(X,X) &= V_t \left(\Ric_g(X,X)-\Lambda-|E|^2 \right) \\
  &= V_t \left( K_{t} -\Lambda-|E|^2 \right) +\langle R(X,N_t)N_t,X\rangle_g V_t\\
  &= V_t \left( K_{t} -\Lambda-|E|^2 \right) - \hess_{\Sigma_t} V_t(X,X),
\end{align*}
Since each $\Sigma_t$ is totally geodesic, it holds $\hess V_t=\hess_{\Sigma_t} V_t$. Thus we obtain
\begin{align}
  \hess_{\Sigma_t} V_t & = \frac{1}{2}(K_t - \Lambda-|E|^2)V_t g_{\Sigma_t}, \label{hessi} \\
 \Delta_{\Sigma_t}V_t & = (K_t - \Lambda-|E|^2)V_t. \label{eq:laplacian}
\end{align}
Therefore, \( (\Sigma_t, g_{\Sigma_t}) \) is a Riemannian manifold that admits a positive global solution for the equation
\begin{align}
  \label{eq:conformal_hessian}
  \hess_{\Sigma_t} V_t = \frac{\Delta_{\Sigma_t} V_t}{2} g_{\Sigma_t}.
\end{align}

Henceforward, we will omit the subscript $t$.

Recall that $\diver \hess_{\Sigma} V=d(\Delta_{\Sigma}V)+KdV.$ Thus, using \eqref{eq:conformal_hessian} and \eqref{eq:laplacian} we obtain,
\begin{align*}
  \frac{1}{2} d\left( \Delta_{\Sigma} V \right) & = d\left( \Delta_{\Sigma} V \right) + K dV \\
  \Rightarrow 0 = 2K dV + d\left( \Delta_{\Sigma} V \right) & = 3K dV + V dK - d\left[ \left( |E|^2 + \Lambda \right) V \right]
  \end{align*}
So, multiplying both sides by $V^2$ we get,
\begin{align*}
  d\left( K V^3 \right) & = V^2 d\left[ \left( |E|^2 + \Lambda \right) V \right] = V^3 d\left( |E|^2 + \Lambda \right) + \left( |E|^2 + \Lambda \right) \frac{d(V^3)}{3} \\
                        & = d\left[ \left( \frac{|E|^2 + \Lambda}{3} \right) V^3 \right] + \frac{2}{3} V^3 d\left( |E|^2 + \Lambda \right).
\end{align*}
On the other hand, using \eqref{eq:trick_galloway} it follows that,
\begin{align*}
  V^3 d\left( |E|^2 + \Lambda \right) & = V^3 d\left( \alpha^2 V^{-4} \right) = - 4 \alpha^2 V^{-2} dV \\
                                      & = 4 d\left( \alpha^2 V^{-1} \right) = 4 d\left( |E|^2 V^3 \right).
\end{align*}
Therefore, combining theses results we obtain the following integrability condition:
\begin{equation*}
d\left[V^3\left(K - \frac{\Lambda}{3} - 3|E|^2\right)\right]=0.
\end{equation*}
Thus, there exists a constant $c$ such that 
\begin{equation}\label{eq:gauss_curvature}
K = c V^{-3} + 3|E|^2 + \frac{\Lambda}{3}.
\end{equation}

Suppose $\partial\si \neq \emptyset$. Since $V_{\partial\si} \equiv 0$, by \eqref{eq:trick_galloway} and \eqref{eq:gauss_curvature} we conclude that $\alpha = 0$ and $c = 0$, respectively. So $|E|V^2 = 0$ and $V^3\left(K - \frac{\Lambda}{3} - 3|E|^2\right) = 0$. Since $V \neq 0$ on $\si\setminus\partial\si$ we obtain $|E| \equiv 0$ and $K = \frac{\Lambda}{3}$. So, by \eqref{eq:laplacian}
\begin{equation}\label{eq:lapl.pot}
\Delta_{\Sigma}V + \frac{2\Lambda}{3}V = 0.
\end{equation}
Multiplying \eqref{eq:lapl.pot} by $V$, integrating by parts on $\si$ and using that $V_{\partial\si} \equiv 0$, one concludes that $\Lambda > 0$. Therefore $V \neq 0$ satisfies the static equations, so as in item ii) of Lemma \ref{lemma:propert} one can prove that $\partial\si = V^{-1}(0)$ is a (piece-wise) geodesic. Since $\si$ has constant positive curvature we conclude that $\si$ is isometric to a domain bounded by geodesics in the round sphere $\sw$ of constant curvature $\frac{\Lambda}{3}$.

If $\si$ is closed and $E \equiv 0$, we can proceed as in Appendix B of \cite{A}. However, we will fix a gap in \cite{A} (which we explain below) in the case $\Lambda > 0$. This correction was suggested to us by L. Ambrozio in a private communication, for which he has our cordial thanks. We want to prove that $V$ is constant on $\si$. Assume that $V$ is not constant. As proved in \cite{KuhnelConformalTransformationsEinstein1988}, $V$ has precisely two non-degenerated critical 
points, denoted by $p_1$ and $p_2$ (which are respectively, the point of minimum  and  maximum), $V$ is increasing and only depends on the distance to $p_1$. Moreover, $\si$ is a topological sphere, and the metric on $\Sigma\setminus \{p_1, p_2\}$ can be written as 
$$ g=d u^{2}+\left(\frac{V^{\prime}(u)}{V^{\prime \prime}(0)}\right)^{2} d \theta^{2},$$
 where $u \in\left(0, u_{0}\right)$ is the distance to $p_{1}$  and $\theta$ is a $2\pi$-periodic variable. Using that $\hess V=V''g,$  equation \eqref{hessi} may be expressed as
\begin{align}\label{newhessian}
2V''= (K-\Lambda)V = cV^{-2} -\frac{2}{3}\Lambda V.
\end{align}
Since 
\begin{align*}
\left[\frac{(V')^2}{2}+\frac{c}{2V}+\frac{\Lambda}{6}V^2\right]' = V'\left[V''-\frac{c}{2V^2}+\frac{\Lambda}{3} V\right]=0,
\end{align*}
there exists a real constant $d$ such that 
\begin{equation*}
(V')^2+\frac{c}{V} +\frac{\Lambda}{3}V^2=d, 
\end{equation*}
on $(0,u_0).$ That is equivalent to the following equation:
\[
\frac{\Lambda}{3}V^3 - dV + V(V')^2 + c=0.
\]
Using that $V'(0)=V'(u_0)=0,$ we see that $A=V(0)$ and $B=V(u_0)$ are roots of the polynomial $P(x)=(\Lambda/3)x^3-dx+c.$ The other root will be denoted by $C$.

Moreover, applying the Gauss-Bonnet theorem we have
$$4\pi = \int_{\Sigma} K d\sigma=\frac{2 \pi}{V^{\prime \prime}(0)}\left[-\frac{c}{2V^2}+\frac{\Lambda}{3} V\right]_{0}^{u_0}.$$
In \cite{A} it was claimed that the last equation together with \eqref{newhessian} implies $c= 2\frac{\Lambda}{3}V^3(u_0)$ (in \cite{A}, $\Lambda = 3$). However the correct conclusion is
$$
2V''(0)=-\frac{c}{2V^2(u_0)}+\frac{\Lambda}{3} V(u_0)+\frac{c}{2V^2(0)}-\frac{\Lambda}{3} V(0).
$$
By \eqref{newhessian}, we conclude that
\begin{equation}\label{GBF}
2(V(u_0)+V(0))=\frac{3c}{\Lambda}\left(\frac{1}{V(0)^2}+\frac{1}{V(u_0)^2}\right).
\end{equation}

However, by Vieta's formulas we have $A+B+C = 0$ and $\frac{3c}{\Lambda} = -ABC = (A + B)AB$. Thus we can rewrite \eqref{GBF} as
\begin{eqnarray*}
2(A + B) = (A + B)AB\left(\frac{1}{A^2}+\frac{1}{B^2}\right).
\end{eqnarray*}
Thus, $2AB=A^2+B^2$, and hence $A = B.$ But this a contradiction with the fact that $V$ is increasing.
%If $V$ is a positive constant solution for \eqref{eq:conformal_hessian} on $\Sigma_0$, then   \eqref{eq:trick_galloway} implies that $|E|$ is constant on $\Sigma_0$. So, from  \eqref{eq:laplacian}, it follows that $(\Sigma_0, g_{\Sigma_0})$ has constant Gaussian curvature $K = |E|^2 + \Lambda$, and performing a reparametrization we conclude that $(U, g\vert_U)$ is isometric to the cylinder $\left( [0,s^{\ast}) \times \Sigma_0, ds^2 + g_{\Sigma_0} \right)$.

Now, assume $\si$ is closed and $E$ is not identically zero. It follows from \eqref{eq:trick_galloway} that \( |E| > 0 \) everywhere on \( U \). Since  \( 0 = \frac{\partial }{\partial_t} H_t = \Delta_{\Sigma}V + \Ric_g(N,N) V \), we have
\begin{align*}
  \Ric_g(N,N) V = -\Delta_{\Sigma}V.
\end{align*}
On the other hand, \( \Ric_g(N,N)V = \hess_{g} V(N,N) + \Lambda - |E|^2 \). Thus, by \eqref{eq:trick_galloway},
\begin{align*}
\hess_g V(N, N) = -\Delta_{\Sigma}V + \alpha^2 V^{-4} - \Lambda.
\end{align*}

Using that \( N = V^{-1} \partial_t \), we obtain \( \nabla_{N} N = - V^{-1} \nabla_{\Sigma}V \) and
\begin{align*}\label{K2}
  \hess_g V(N, N) = N(N(V)) - \nabla_{N} N(V)  = V^{-1} \left( \frac{d^2}{dt^2} \big(\log \sqrt{\alpha}\big) + |\nabla_{\Sigma}V|^2 \right).
\end{align*}
So, we get
\begin{align*}
  \frac{d^2}{dt^2}  \big(\log \sqrt{\alpha}\big) + |\nabla_{\Sigma}V|^2
    = - V \Delta_{\Sigma}V + \alpha^2 V^{-3} - \Lambda V.
\end{align*}
Then, for any \( X \in \mathfrak{X}(\Sigma) \), we have
\begin{equation}
  X\left( |\nabla_{\Sigma}V|^2 \right) + V X\left( \Delta_{\Sigma}V \right)
  + \left( \Delta_{\Sigma}V + 3 \alpha^2 V^{-4} + \Lambda \right) X(V) = 0. \label{eq:differentiation}
\end{equation}
Using \eqref{eq:conformal_hessian}, it follows that
\begin{align*}
  X\left( |\nabla_{\Sigma}V|^2 \right) = (\Delta_{\Sigma}V) X(V).
\end{align*}
Combining equations \eqref{eq:trick_galloway}, \eqref{eq:laplacian} and \eqref{eq:gauss_curvature},
we obtain
\begin{align*}
  \Delta_{\Sigma}V = c V^{-2} + 2\alpha^2 V^{-3} - \frac{2 \Lambda}{3} V.
\end{align*}
Substituting in the equation \eqref{eq:differentiation} and multiplying by \( - V^4 \),
we get
\begin{align*}
  \left(2\Lambda V^5 - \Lambda V^4 + 2\alpha^2 V - 3 \alpha^2 \right) X(V) = 0.
\end{align*}

The polynomial $\psi(y) = 2\Lambda y^5 - \Lambda y^4 + 2 \alpha^2 y^2 - 3 \alpha^2$ has at most five zeros, which are then isolated. 
Thus, either $X(V) = 0$ in $\si$, $\forall\, X\in \mathfrak{X}(\Sigma)$, or there exist $X \in \mathfrak{X}(\Sigma)$ and $p \in \si$ such that $X(V)|_p \neq 0$, so $\psi(V) = 0$ in some connected neighborhood $\mathcal{U}$ of $p$, which implies that $V|_{\mathcal{U}}$ is constant. Then $V$ is locally constant, and since $\si$ is connected, $V$ must be constant along $\si$. By \eqref{eq:trick_galloway}, it follows that $|E|$ is constant on $\si$. Moreover, $|E(t,\cdot)| = f$ does not depend on $t$, so $|E|$ is constant on $\Phi\big([0,\delta)\times \Sigma_0\big)$. 

Since $|E|$ is a positive constant, it follows from \eqref{eq:trick_galloway} that $\alpha(t) > 0$. The reparametrization $s=\int_0^t \sqrt{\alpha(r)}\,dr$ allow us to conclude that $(U, g\vert_U)$ is isometric to the Riemannian product
$$\left( [0, s^{\ast}) \times \Sigma_0, \frac{1}{|E|}\,ds^2 + g_{\Sigma_0} \right).$$
\end{proof}

\subsection{Topological consequences}
The proof of the following result goes along the same lines as in \cite[Proposition $15$]{A}. However, we write the proof here for completeness.

\begin{prop}\label{injective}
 Let $(M^3,g,V,E)$ be a compact electrostatic system such that $V^{-1}(0) = \partial M \neq \emptyset$. Then, the homomorphism $i_*:\pi_1(\partial M)\to\pi_1(M)$, induced by the inclusion $i:\partial M \to M$, is injective.
\end{prop}

\begin{proof}
Let $[\gamma]\in \pi_{1}(\partial M)$, where $\gamma$ is a smooth embedded closed curve, and assume  $i_{*}[\gamma] = 0$ in $\pi_{1}(M)$. 
Let $\mathcal{F}_M$ denote the set of all immersed disks in $M$ whose boundary is $\gamma.$ We define 
\begin{equation}\label{Plateau}
\mathfrak{A}(M,g)=\inf_{\Sigma\in\mathcal{F}_M}|\Sigma|.
\end{equation}

Since $\partial M$ is mean convex it follows from a classical result of Meeks and Yau  \cite[Theorems 1 and 2]{MeeYau} that there exists an immersed minimal disk $\Sigma_0$ in $M$ such that $|\Sigma_0|= \mathfrak{A}(M,g),$ where $\Sigma_0$ is  either  contained in $\partial M$ or properly embedded in $M$. 

Suppose that $\Sigma_0$ is   properly embedded in $M,$  since otherwise   $\Sigma_0\subset \partial M$ would imply  that $[\gamma]=0$ in $\pi_{1}(\partial M)$.  
Consider the smooth flow $\{\si_t\}_{t\in[0,\delta)}$, defined by \eqref{flow},   starting at $\Sigma_0$ with normal speed $V$ and so that  each $\partial \Sigma_t = \gamma$. According to the Proposition \ref{goodlemma}, we have $|\Sigma_t|\leq |\Sigma_0|.$ On the other hand, the opposite inequality also holds, since $\Sigma_0$ is a solution to the Plateau problem. Thus, $|\Sigma_t| = |\si_0|$, $\forall \ t\in[0,\delta)$, so by item ii) of Proposition \ref{goodlemma}, each $\si_t$ isometric to a hemisphere $\sw_{+}$ with constant Gaussian curvature \( \frac{\Lambda}{3} \).

Let $T>0$  be the maximal time of existence and smoothness of the flow defined by \eqref{flow}. Suppose $T < \infty$. 
First, observe that since $(M\setminus \partial M,V^{-2}g)$ is complete, the surfaces $\Sigma_t\setminus\partial \Sigma_t$ never touch $\partial M$ in finite time. Now, assume that $T<+\infty$, and consider a sequence $t_{i}\rightarrow T$. Each $\Sigma_{t_{i}}$ is a solution of \eqref{Plateau}, so by standard compactness of stable minimal surfaces of bounded area, a subsequence of $\{\Sigma_{t_{i}}\}$ converges to another solution of \eqref{Plateau}. Hence, it would be possible to continue the flow beyond $T$, which is a contradiction.
 
Flowing in the opposite normal direction and using again the Lemma \ref{goodlemma} we conclude that $(M\setminus\partial M,g)$ is isometric to $(\st_{+}\setminus\partial \st_{+},g_{can})$. In particular, $M$ is diffeomorphic to $\st_+$ and so $[\gamma]=0$ in $\pi_{1}(\partial M)$.
\end{proof}

Let us recall some definitions about the topology of $3$-manifolds. 

\begin{definition}
A \textit{compression body} is a $3$-manifold $\Omega$ with boundary with a particular boundary
component $\partial_+\Omega=\Sigma\times\{0\}$ such that $\Omega$ is obtained from $\Sigma\times
[0,1]$ by attaching $2$-handles and $3$-handles, where no attachments are performed along
$\partial_+\Omega=\Sigma\times\{0\}$.

A compression body with only one boundary component, \textit{i.e.} $\partial
\Omega = \partial_+\Omega$, is called a \textit{handlebody}. A handlebody can also be seen as a closed
ball with $1$-handles attached along the boundary.
\end{definition}

We can now state a result that characterize the topology of a certain class of electrostatic systems.

\begin{theo}\label{topology}
Let $(M^3,g,V,E)$ be a compact electrostatic system, such that $V^{-1}(0) = \partial M$ and $\Lambda + |E|^2 > 0$. Then 
\begin{enumerate}
\item If $\partial M$ contains an unstable component, then the number of unstable components of $\partial M$ is equal to one and $M$ is simply connected.
\item Each connected component of $\partial M$ is diffeormorphic to a $2$-sphere.
\end{enumerate}
\end{theo}

\begin{proof}
First, since $R_g = 2\Lambda + 2|E|^2 > 0$, by \cite{SY3} each stable component of $\partial M$ is homeomorphic to a $2$-sphere. The fact that unstable components are also spheres will follow from item (1), which we prove below.\\

\noindent
{\bf Step 1: $(\interior M,g)$ does not contain embedded closed minimal surfaces whose orientable 2-cover is stable.}\\

First, suppose $\interior M$ contains an orientable embedded closed stable minimal surface $\si$. Write $\partial M = \partial M_s\cup\partial M_u$, where $\partial M_s$ denotes the union of the stable components, and $\partial M_u$ denotes the union of the unstable components. Let $\Omega$ be a component of $M\setminus \si$ such that $\Omega\cap\partial M_u \neq \emptyset$. Minimize area in the $\mathbb{Z}$-homology class of $\si$ inside $\Omega$ (see \cite[Sections 5.1.6 and 5.3.18]{Fed}). We have two possibilities:
\begin{enumerate}
    \item $\Omega$ {\it contains no component of} $\partial M_s$.
   
\noindent    
In this case the surface obtained by minimization is either equal to $\si$ or some component of it is contained in the interior of $\Omega$. In particular the surface is contained in $\interior M$.\\
    
    \item $\Omega$ {\it contains some component of} $\partial M_s$.

\noindent
Since $\partial M_u$ and $\partial M_s$ are homologous, $\si$ can not be homologous to $\partial M_s\cap\partial\Omega$. Thus the surface obtained by minimization has at least one component disjoint from $\partial M_s$.
\end{enumerate} 
In any case, we obtain a surface $S \subset \interior M$ which minimizes area locally.

Suppose $S$ is orientable. Let $N$ be a unit normal vector to $S$. Consider the flow by \eqref{flow}. It follows from Proposition \ref{goodlemma} that the function $t \mapsto |S_t|$ is non-increasing, where $S_t$ denotes the surfaces along the flow. In particular, $|S_t| \leq |S|$. On the other hand, by the minimization property $|S| \leq |S_t|$, for $t \in (0,t_0]$. Then, $|S_t| = |S|$, for $t \in (0,t_0]$. So, the statement $i)$ in Proposition \ref{goodlemma} holds true. 

Let $T^{*} > 0$ be the maximal time in which the flow (\ref{flow}) exists and is smooth. Observe that the surfaces $S_t$ never touch the boundary of $M$ in finite time (since this would imply that $(M\setminus \partial M,V^{-2}g)$ is incomplete). We have two possibilities, either $T^{*}=+\infty$ or $T^{*}$ is finite. Suppose the second case happens and consider a sequence $t_{j}\rightarrow T^{*} $. Then the corresponding surfaces of the flow $S_{t_{j}}$ are area minimizing surfaces, so as in the previous theorem a subsequence converges to another minimizing surface $S_{T^{*}}$. This surface  necessarily is non-orientable, otherwise we could continue the flow beyond $T^{*}$, and this contradicts the definition of $T^{*}$. Moreover, since the surfaces $S_t$, $0 \leq t < T^{*}$, are stable, and $(M,g)$ has positive scalar curvature, by \cite{SY3} each $S_t$ is a $2$-sphere. Hence $S_{T^{*}}$ is a $\mathbb{RP}^2$ and has a tubular neighborhood inside $M$ diffeomorphic to $\mathbb{RP}^3$ minus a ball.

Now, we flow in the direction of the opposite normal. Defining  $T_{*} < 0$ as the maximal time in which the flow exists and is smooth, the argument works as before. In the end, we have obtained an isometric embedding in $(M,g)$ of a manifold $\mathcal{N}$ which is diffeomorphic to $(-\infty,+\infty)\times \mathbb{S}^2$ (if $T_{*} = -\infty$ and $T^{*} = +\infty$), $\mathbb{RP}^3$ minus an open ball (if $T_{*} > - \infty$ and $T^{*} = +\infty$, or $T_{*} = - \infty$ and $T^{*} < +\infty$), or  $\mathbb{RP}^3\#\mathbb{RP}^3$ (if both $T_{*}$ and $T^{*}$ are finite). Moreover the induced metric in $\mathcal{N}$ is complete, so necessarily $\mathcal{N} = M$. However, $M$ is compact with non-empty boundary, while in the first two possibilities $\mathcal{N}$ is non-compact without boundary, and in third one $\mathcal{N}$ is closed, so we obtain a contradiction. Therefore $S$ can not be orientable.

Suppose $S$ is non orientable. Its orientable 2-cover is stable. We can pass to a double covering $\tilde M$ of $M$ such that the lift $\tilde S$ of $S$ is a connected closed orientable minimal surface (see Proposition 3.7 in \cite{Zhou}). Let $\Phi$ be the covering map. Defining $\tilde{g} = \phi^{*}g$, $\tilde{V}(p) = V\bigl(\Phi(p)\bigr)$ and $\tilde{E}(p) = E\bigl(\Phi(p)\bigr)$, we have that $\Phi$ is a local isometry between $(\tilde M,\tilde g)$ and $(M,g)$ and $(\tilde M,\tilde g,\tilde V,\tilde{E})$ also satisfies the equations of an electrostatic system. By the hypothesis, $\tilde{S}$ is stable, so we can find a contradiction as before. Therefore, $\interior M$ does not contain any orientable embedded closed stable minimal surface.

Now, suppose $M$ contains a non orientable embedded closed minimal surface $\si$ whose orientable 2-cover is stable. We can pass to a double cover as in the last paragraph, and proceed as before to obtain a contradiction.\\

\noindent
{\bf Step 2: The number of unstable components of $\partial M$ is equal to one.}\\

We proceed as in Lemma $3.3$ of \cite{LN}. Write $\partial M = \partial_u M\cup\bigl(\cup_{i=1}^{k}\partial_i M\bigr)$, where each $\partial_i M$ is stable, and $\partial_u M$ denotes the union of the unstable components.  Let $\Sigma^{1}, \ldots, \Sigma^{\ell}$ be the connected components of $\partial_u M$.  

We can apply the main result of \cite{MSY} to minimize the area in the isotopy class of $\Sigma^{1}$ in $M$. 
By Section 3 and Theorem 1 of \cite{MSY}, after possibly performing isotopies and finitely many $\gamma$-reductions (a procedure that removes a submanifold homeomorphic to a cylinder and adds two disks in such a way that the cylinder and two disks bound a ball in $M$) one obtains from $\Sigma^{1}$ a surface $\tilde{\Sigma}^{1}$ such that each component of $\tilde{\Sigma}^{1}$ is a parallel surface of a connected minimal surface, except possibly for one component that may be taken to have arbitrarily small area. 

Since $\gamma$-reduction always preserves the homology class and any closed minimal surface whose orientable 2-cover is stable is one of $\partial_i M'$s, there exist positive integers $n_1,\dots, n_k$ such that 
\begin{equation}\label{homologous}
  [\Sigma^{1}] = [\tilde{\Sigma}^{1}]=
  \sum_{i=1}^{k}n_i [\partial_i M] \text{ in }H_2(M,\zz), 
  \end{equation}
where we used that the fact that surfaces of area small enough must be homologically trivial.

Using the long exact sequence for the pair $(M,\partial M)$, we have exactness of 
$$H_3(M,\partial M)\overset{\partial}{\too} H_2(\partial M,\zz)\overset{\iota_*}{\too} H_2(M,\zz).$$
Observe that since $M$ is connected, $\ker i_*$ must be generated by 
\[\partial [M]=\sum_{i=1}^{\ell} [\Sigma^{i} ] -\sum_{i=1}^{k}[\partial_i M].\]
Here we should remark that $\Sigma^{i}$ and $\partial M$ are oriented using the outward normal in $M$.  
Since \eqref{homologous} imply that 
$$[\Sigma^{1}] - \sum_{i=1}^{k}n_i [\partial_i M]\in \ker i_*,$$ 
we conclude that $\partial_u M$ is connected, indeed equal to $\Sigma^{1}$. In particular,
\begin{equation}\label{hom.rel}
[\Sigma^{1}] =\sum_{i=1}^{k}[\partial_i M] \text{ in } H_2(M,\zz).
\end{equation}

\noindent
{\bf Step 3: Denote by $\mathfrak{I}$, the isotopy class of $\partial_u M$. Then, there exist positive integers $m_1,\ldots,m_k$ such that}
$$\inf_{S \in \mathfrak{I}}|S| = \displaystyle\sum_{i = 1}^{k} m_i|\partial_i M|.$$

We will use the same notation of  Step $2$. Fix $j \in \{1,\ldots,k\}$. Since each component of $\tilde{\Sigma}^{1}$ is either isotopic to one of the $\partial_i M$'s (with some orientation) or is null homologous, and since there are no relations among $[\partial_i M]$ in $H_2(M,\zz)$, the equations \eqref{homologous} and \eqref{hom.rel} imply that at least one component of $\tilde{\Sigma}^{1}$ is isotopic to $\partial_j M$. The conclusion follows using the description of the minimization process in \cite{MSY}.\\

\noindent
{\bf Step 4: $M$ is a compression body.}\\

For each $i \in \{1,\ldots,k\}$, denote $h_i = g|_{\partial_i M}$ and consider the Riemannian metric $g_i = dt^2 + (\cos^2 t)h_i$ in $[0,\pi/2)\times\partial_i M$. A calculation shows that the surface $S_i = \{0\}\times\partial_i M$ is totally geodesic and $g_i$ extends to a Riemannian metric on the $3$-ball $B^3$, which we still denote by $g_i$.

Now, consider the Riemannian manifold $(\widehat{M},\widehat{g})$ obtained by gluing $(B^3,g_i)$ to $(M,g)$, for each $i \in \{1,\ldots,k\}$, where each $\partial_i M$ is identified with $S_i$. Since these two surfaces are totally geodesic, the metric $\widehat{g}$ is $C^{1,1}$. Also, $\partial\widehat{M} = \partial M_u$, so the boundary of $\widehat{M}$ is connected. We will prove that $\widehat{M}$ is a handlebody, and since $M$ is obtained by removing open $3$-balls from the interior of $\widehat M$, it follows that $M$ is a compression body.

By Proposition $1$ of \cite{MSY}, if the infimum of the area in the isotopy class of $\partial M_u$ inside $\widehat M$ is zero, then $\widehat M$ is a handlebody. So, we are going to prove that indeed that this infimum is zero. We should remark that as pointed out in  Section $2$ of \cite{MA}, the results of \cite{MSY} still hold if the metric is $C^{1,1}$. 

As previously proved the infimum of the area in the isotopy class of $\partial M_u$ inside $M$ is equal to
$\sum_{i = 1}^{k} m_i|\partial_i M|.$
Thus, by Remark $3.27$ in \cite{MSY} for any sufficiently large positive integer $j$, there is $S_j = \cup_{i = 0}^{k} S^{(i)}_{j}$ obtained from $\partial M_u$ via isotopy and a series of $\gamma$-reductions such that the following holds:
\begin{itemize}
\item the infimum of the area in the isotopy class of $S^{(0)}_{j}$ is equal to zero;
\item $S^{(0)}_{j}\cap\left(\cup_{i = 1}^{\ell} S^{(i)}_{j}\right) = \emptyset$;
\item for $i > 0$ we have
\begin{equation*}
S^{(i)}_{j} = \left\{
\begin{array}{rl}
&\cup_{r=1}^{n_i}\big\{x \in M;\, \mathrm{dist}(x,\partial_i M) = \frac{r}{j}\big\}, \ \textrm{if}\ m_i = 2n_i\\\\
&\partial_i M\cup\Big(\cup_{r=1}^{n_i}\big\{x \in M;\, \mathrm{dist}(x,\partial_i M) = \frac{r}{j}\big\}\Big), \ \textrm{if}\ m_i = 2n_i + 1.
\end{array} \right.
\end{equation*}
\end{itemize}

Each $\partial_i M$ bounds a $3$-ball in $\widehat{M}$, hence each component of $S^{(i)}_{j}$ is isotopic to surfaces of arbitrarily small area, for $i > 0$. We can join the components of $S_{j}$ by tubes of very small area, obtaining thus a surface $\widehat{S}_{j}$ isotopic to $\partial M_u$. Therefore the infimum of the area the isotopy class of $\partial M_u$ in $\widehat{M}$ is equal to zero. So, the conclusion follows.\\

\noindent
{\bf Step 5: $M$ is homeomorphic to a closed $3$-ball minus a finite number of disjoint open $3$-balls.}\\

Since $M$ is a compression body, the homomorphism $i_*:\pi_1(\partial M)\to\pi_1(M)$ induced by the inclusion $i:\partial M \to M$ is surjective. On the other hand, by Proposition \ref{injective}, $i_*$ is also injective. Hence $i_*$ is an isomorphism, and since $M$ is a compression body this is only possible if $M$ is homeomorphic to a closed $3$-ball minus a finite number of disjoint open $3$-balls.
\end{proof}

\section{Min-Max Characterization of unstable horizons}\label{MMC}

\subsection{Min-max constructions}

We begin by describing the heuristic idea behind the min-max theory. Suppose we have a smooth function $f$ defined on some topological space (where the concept of "smoothness" is available), which has two points of strict local minimum $q_1$ and $q_2$. One then expect to find a third critical point of saddle type by a \emph{mountain pass} argument, which we will explain now. Fix a continuous curve $\gamma_0$ joining $q_1$ and $q_2$, and consider the family $\Gamma$ of all continuous curves $\gamma$ which join $q_1$ and $q_2$, and which can be deformed continuously into each other and into $\gamma_0$ as well. In particular $\gamma_0$ belongs to $\Gamma$. Now, define the quantity
$$\lambda = \inf_{\gamma \in \Gamma}\max_{t}f\big(\gamma(t)\big).$$
The next step is to take a sequence $\{p_i = \gamma_i(t_i)\}_{i}$ such that $f(p_i) \to \lambda$, where $\gamma_i \in \Gamma$, and try to obtain a subsequence which converges to a critical point $p$, which must then necessarily satisfy $f(p) = \lambda$.

In our setting, we have a Riemannian 3-manifold, the function is the area functional and the space consists of closed surfaces and degenerated sets with "zero area" (e.g. points and curves) embedded in the manifold.  Critical points of the area correspond to minimal surfaces and points of minimum correspond to surfaces which locally minimize area and degenerated sets as well. Consider a compact domain $\Omega$ between two points of minimum. To employ the idea described above we sweep $\Omega$ out by a smooth one-parameter family of surfaces $\{\Sigma\}_{t \in [0,1]}$. Then, we consider the class $\Pi$ of all families $\{\Psi_t(\Sigma_t)\}_{t \in [0,1]}$, for some smooth one-parameter family of diffeomorphisms $\Psi_t$, all of which isotopic to the identity, and define the quantity  
$$W = \inf_{\Pi}\max_{t \in [0,1]}|\Psi_t(\Sigma_t)|.$$

As before, we would like to take a sequence of slices $\{\Psi_{t_i}(\Sigma_{t_i})\}_i$ such that $|\Psi_{t_i}(\Sigma_{t_i})| \to W$, and the hope is to prove that some subsequence converges (in some sense) to a minimal surface $\Sigma$ whose area is equal to $W$. The notion of convergence requires a topology in our space, and a natural one is the $C^k$-topology, $k \geq 2$. The technical difficulty is that in this topology a control in the area is in general not sufficient to guarantee convergence. A way to overcome this is to use the machinery of \emph{geometric measure theory}, where there are notions of generalized surfaces (e.g. currents and varifolds) and weak convergence, which provide compactness and regularity results. In the following we present the construction necessary to obtain minimal spheres in this setting.

Let $\Omega$ be a connected compact $3$-manifold with boundary, subset of an oriented complete $3$-manifold $(M,g)$. Here $\mathcal{H}^2$ denotes the Hausdorff measure of dimension $2$. Recall that if $\Sigma$ is a smooth surface embedded in $\Omega$, then $\mathcal{H}^2(\Sigma)$ is equal to the area of $\Sigma$. The Almgren map $\mathcal{A}$ (see \cite{Alm1}) associates to a continuous family of surfaces $\{\Sigma_t\}$ a $3$-dimensional integral current $\mathcal{A}(\{\Sigma_t\})$ (see Appendix \ref{appendix} for the definition of currents). The $\mathbf{F}$-norm for varifolds is defined in \cite{P}. For the basic theory of currents and varifolds see \cite{Fed}.

\begin{definition}
Let $\{\Sigma_t\}_{t\in[a,b]}$ be a family of closed subsets of $\Omega$ with finite $\mathcal{H}^2$-measure. We say that $\{\Sigma_t\}$ is a \emph{sweepout by spheres} of $\Omega$ if there are a finite subset $T$ of $[a, b]$ and a finite set of points $P$ in $\Omega$ such that
\begin{enumerate}
\item for all $t\in (a,b)\setminus T$, $\Sigma_t$ is a union of disjoint smooth embedded $2$-spheres in the interior of $\Omega$;
\item for $t \in T$, $\si_t\setminus P$ is a union of smooth embedded $2$-spheres minus points in $\Omega$;
\item $\Sigma_t \to \si_{t_0}$ in the Hausdorff topology whenever $t \to t_0$;
\item $\si_t$ varies smoothly in $[a,b]\setminus T$, and if $t_0 \in T$, then $\si_t$ converges smoothly to $\si_{t_0}$ in $\Omega\setminus P$ as $t\to t_0$;
\item there is a partition $(A,B)$ of the components of $\partial \Omega$ such that, $\Sigma_a=A\cup C_a$, $\Sigma_b=B\cup C_b$, where $\mathcal{H}^2(C_a) = 0 = \mathcal{H}^2(C_b)$. Moreover $\Sigma_t$ converges to $\Sigma_a$ (resp. $\Sigma_b$) in the $\mathbf{F}$-norm, as $t\to a$ (resp. $b$);
\item if $[[\Omega]]$ denotes the $3$-dimensional integral current given by $\Omega$ with its orientation, then we have $\mathcal{A}(\{\Sigma_t\})= [[\Omega]]$.
\end{enumerate}
\end{definition}

For a sweepout $\{\si_t\}_{t\in[a,b]}$, we define the quantity
$${\bf L}\left(\{\si_t\}\right)=\displaystyle\max_{t\in[a,b]}\boH^2(\si_t).$$

\begin{figure}[!h]
\centering
\includegraphics[scale=0.033]{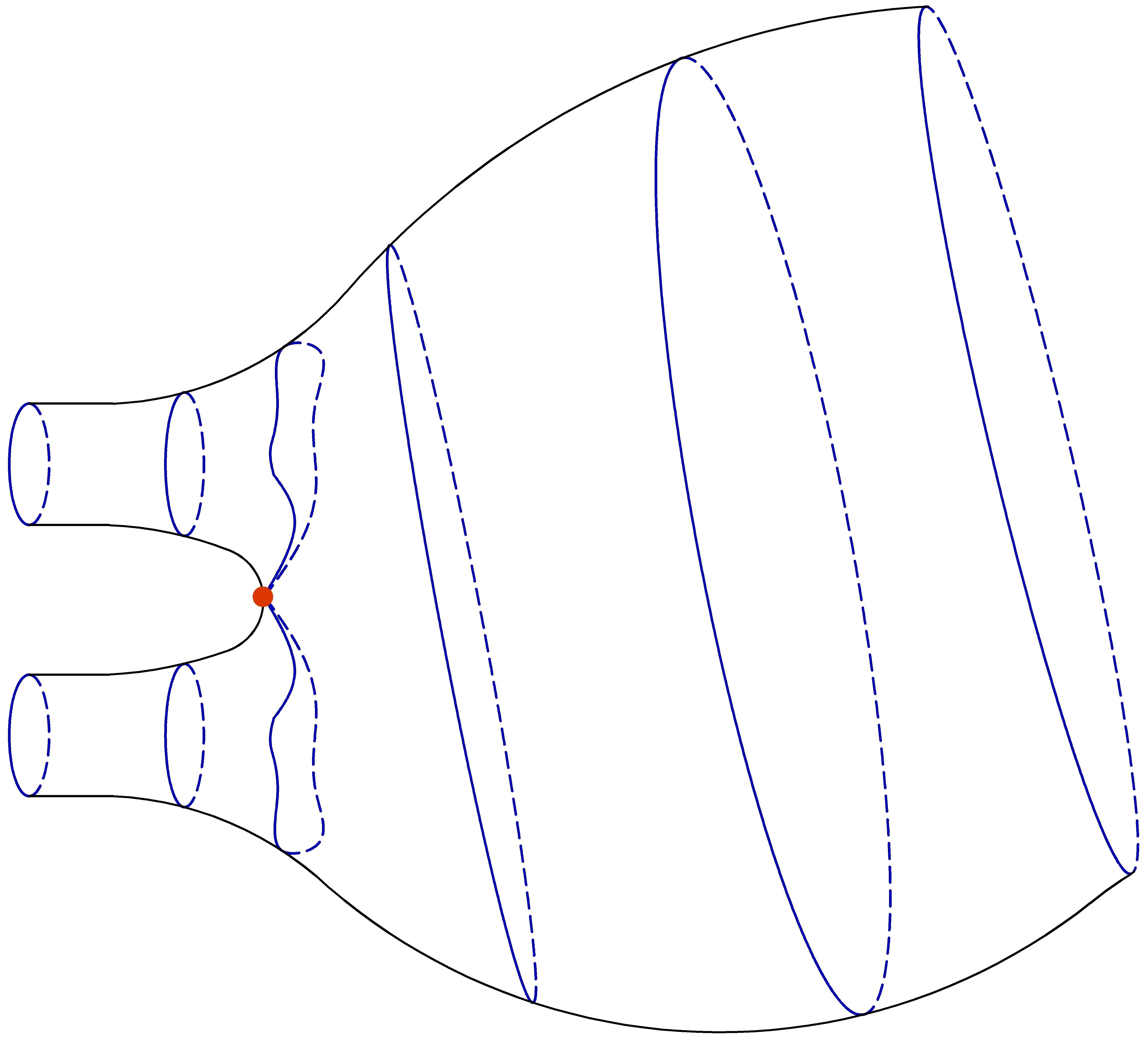}
\caption{A sweepout by spheres. The point in red belongs to the singular set $P$.}
\label{fig:sweepout}
\end{figure}

An example of a sweepout is given by the level sets of $x_4$ in the sphere $\mathbb{S}^3$. In this case, $[a,b] = [-1,1]$, $\Omega$ has no boundary, and $\Sigma_{-1},\Sigma_{1}$ are points. Also, if $\Omega \approx \mathbb{S}^2\times[0,1]$, then $\big\{\Sigma_t = \mathbb{S}^2\times\{t\}\big\}_{t \in [0,1]}$ defines a sweepout. Finally, if $\Omega$ homeomorphic to a closed $3$-ball minus a finite number of disjoint open $3$-balls, then $\Omega$ admits a sweepout by spheres, however, in this case some of the slices can have singularities or be disconnected (see figure \ref{fig:sweepout}).

Let $\Pi$ be a collection of sweepouts. Denote by $\Diff_0$ the set of diffeomorphisms of $\Omega$ isotopic to the identity map and leaving the boundary fixed. The set $\Pi$ is saturated if for any map $\Psi\in C^\infty([0,1]\times \Omega, \Omega)$ such that $\Psi(t,\cdot)\in \Diff_0$ for all $t$, and for any $\{\Sigma_t\}\in \Pi$, we have $\{\Psi(t,\Sigma_t)\}\in \Pi$. We say that $\Pi$ is generated by a sweepout $\{\Sigma_t\}$, if $\Pi$ is the smallest saturated set containing $\{\Sigma_t\}$. The \textit{width of $\Omega$ associated with $\Pi$} is defined to be 
$$W(\Omega,\Pi) =\inf_{\{\Sigma_t\} \in \Pi}{\bf L}\left(\{\si_t\}\right).$$

Given a sequence of sweepouts $\{\{\Sigma^i_t\}\}_i\subset \Pi$, we say this sequence is \textit{minimizing} if 
$$\displaystyle\lim_{i\to \infty} \max_{t\in[0,1]}\mathcal{H}^2(\Sigma^i_t)=W(\Omega,\Pi).$$ 
If that is the case, let $\{t_i\}$ be a sequence of parameters such that 
$$\mathcal{H}^2(\Sigma^i_{t_{i}}) \to W(\Omega,\Pi),$$ 
then we say that $\{\Sigma^i_{t_{i}}\}_{i}$ is a \textit{min-max sequence}.

In this context we have the following variation of Theorem $10$ in \cite{KLS}.

\begin{theo} \label{smoothminmax}
Let $(M^{3},g)$ be an oriented Riemannian $3$-manifold which does not contain embedded projective planes. Let $\Omega$ be a compact non-empty $3$-submanifold of $M$ such that each component of the boundary $\partial\Omega = \Gamma_0\cup\Gamma_1$ is either a strictly mean convex sphere or a strictly stable minimal sphere.

Assume that there exist a saturated set $\Pi$ generated by sweepouts by spheres, and $N=N(\Pi)<\infty$ such that for any $\{\Sigma_t\}\subset \Pi$, the set $P$ consists of at most $N$ points, and the number of components of $\Sigma_t$ is at most $N$, $\forall\, t$. Suppose that 
$$W(\Omega,\Pi) > \max \{\mathcal{H}^2(\Gamma_0), \mathcal{H}^2(\Gamma_1)\}.$$
Then there exists a min-max sequence $\big\{\Sigma^j_{t_j}\big\}_j$ converging to $\displaystyle\sum_{i=1}^k m_i\Sigma^\infty_i$ as varifolds, where each $m_i$ is a positive integer and $\Sigma^\infty_i\subset \Omega$, $i= 1,\ldots,k$, are disjoint embedded minimal spheres such that
\begin{eqnarray*}
&&\sum_{i=1}^k m_i \mathcal{H}^2(\Sigma^\infty_i) = W(\Omega,\Pi).
\end{eqnarray*}
Moreover, at least one of the components $\Sigma^\infty_i$ is contained in the interior of the domain $\Omega$.
\end{theo}

\begin{proof}       
Let $\Upsilon$ be the union of minimal surfaces in $\partial \Omega$. By the hypothesis on $\Omega$, we can find a small $\delta>0$ such that 
$\Omega_\delta := \Omega \cup \{x\in M; d(x,\Upsilon)\leq \delta\}$
is a strictly mean convex domain and if a closed minimal surface is contained in $\Omega_\delta$ then it is contained in $\Omega$. The saturated set $\Pi$ naturally induces a saturated set $\Pi_\delta$ associated with $\Omega_\delta$. It is then not difficult to check that for $\delta$ small, $W(\Omega_\delta,\Pi_\delta)=W(\Omega,\Pi)$. If $\delta$ is chosen small enough, we can apply the version of the Simon-Smith Theorem proved in \cite[Theorem 2.1]{MaNe} to get the existence of the varifold $V=\displaystyle\sum_{i=1}^k m_i\Sigma^\infty_i.$

By the main result in \cite{Ketgenusbound} the genus of each $\Sigma^\infty_i$ is zero, and the topological assumption rules out the possibility of some component be a projective plane. Thus, each $\Sigma^\infty_i$ is an embedded sphere. The fact that at least one component of the min-max surface is inside $\interior(\Omega)$ was proved in \cite{KLS}.
\end{proof}

\

\subsection{Proofs of Theorems B and C} \label{MMCU}

We first prove the following lemma (see also ~\cite[Lemma 4]{HMM}).

\begin{lemma}\label{lemma:stability}
Let $(M^3,g,V,E)$ be an electrostatic system. Let $\Sigma$ be a  closed, connected, orientable stable minimal surface in $M$. Then: 
\begin{enumerate}
\item Either $V$ does not vanish on $\Sigma$ or $V|_{\si}\equiv 0$. 
\item $\Sigma$ is totally geodesic.
\end{enumerate}
\end{lemma}
\begin{proof}
Let $N$ be a unit normal to $\si$. Since $\langle E,N\rangle^2 \leq |E|^2 $, by the stability inequality, for any $\phi\in C^1(\Sigma)$, 
\begin{align}
\int_\Sigma |\nabla_\Sigma \phi|^2 d\mu
& \geq  \int_\Sigma\left( |A|^2 + \Ric_g(N, N)\right) \phi^2 d\mu \nonumber \\
& \geq  \int_\Sigma\left[ |A|^2 + \Ric_g(N, N) + 2\left( \langle E,N\rangle^2 - |E|^2 \right) \right] \phi^2 d\mu \label{eq:stability}\\
&\geq  \int_\Sigma \left[ \Ric_g(N, N) + 2\left( \langle E,N\rangle^2 - |E|^2 \right) \right] \phi^2 d\mu. \nonumber
\end{align}
This implies that the first eigenvalue of the operator $\Delta_\Sigma + \Ric_g(N, N) + 2\langle E,N\rangle^2 - 2|E|^2$ is non-negative.

Using that $\Sigma$ is minimal and equations \eqref{eq1} and \eqref{eq2} we obtain 
\begin{align}
0 &= \Delta_\Sigma V + \hess V(N, N) - \Delta V \nonumber \\
&= \Delta_\Sigma V +\Ric_g(N, N)V + 2\left(\langle E,N\rangle^2 - |E|^2\right)V,\label{equation:Ric2}
\end{align}
so, either $V$ is the first eigenfunction with the zero eigenvalue, or $V|_{\si} \equiv 0$.

If $V|_{\Sigma} \equiv 0$, then $\Sigma \subset V^{-1}(0)$, hence by Lemma \ref{lemma:propert}, $\si$ is totally geodesic. Now, suppose that $V$ is the first eigenfunction. Then $V$ does not vanish on $\Sigma$. Combining \eqref{equation:Ric2} with \eqref{eq:stability}, we obtain $\int_\Sigma |A|^2 V^2 \ d\mu \leq 0$, which implies $\protect{|A|\equiv 0}$.
\end{proof}

Now we are ready to prove our main results.

\begin{theo}\label{char1}
Consider a complete electrostatic system $(M^3,g,V,E)$, such that $\Lambda + |E|^2 > 0$. Let $\Omega_1, \Omega_2$ be connected maximal regions where $V \neq 0$, such that $\overline\Omega_i$ is compact and let $\si = \partial\Omega_1\cap\partial\Omega_2 \subset V^{-1}(0)$ be unstable. Suppose $\partial\Omega_i\setminus \si$ is either empty or strictly stable, for $i=1,2$. Then $\si$ realizes the min-max width of $(\Omega_1\cup\Omega_2,g)$. In particular, $\si$ has index one.
\end{theo}

\begin{proof}
{\bf Step 1: $\Omega = \Omega_1\cup\Omega_2$ does not contain any embedded closed minimal surface whose orientable 2-cover is stable.}\\ 

By Step $1$ of Theorem \ref{topology}, $\interior\Omega_i$, $i=1,2$, does not contains any embedded closed minimal surface whose orientable 2-cover is stable.

Now, suppose that there is an embedded, connected, closed minimal surface $\tilde{\si} \subset (\Omega,g)$ whose orientable 2-cover is stable and such that $\tilde{\si}\cap\si \neq \emptyset$. Up to taking a double cover of $\Omega$, we can suppose that $\tilde{\si}$ is orientable. By Lemma \ref{lemma:stability} either $|V| >0$ on $\tilde{\si}$ or $V|_{\tilde{\si}}\equiv 0$. Since $\tilde{\si}\cap V^{-1}(0) \neq \emptyset$ the second option holds. Since \( V^{-1}(0) \) is an embedded surface, \( \tilde{\si} \subset V^{-1}(0) \) and $\si$ is connected (by Theorem \ref{topology}) we have necessarily $\tilde{\si} = \si$, which contradicts the fact that $\si$ is unstable. \\

\noindent
{\bf Step 2: There is a sweepout $\{\Sigma_t\}_{t\in[-1,1]}$ of $\Omega$ such that
$\mathbf{L}(\{\Sigma_t\})=|\Sigma|$, $\si_{-1} = \partial\Omega_1\setminus\si$, $\si_1 = \partial\Omega_2\setminus\si$, $\si_0 = \si$ and, for any
$\varepsilon>0$, there is $\delta>0$ such that $|\Sigma_t| \leq |\Sigma|-\delta$, if $|t| \ge \varepsilon$.}\\

First, by Step 5 of Theorem \ref{topology}, $\Omega_i$ is is homeomorphic to a closed $3$-ball minus a finite number of disjoint open 3-balls, $i=1,2$, so $\Omega$ admits a sweepout by spheres. In the remaining of the proof we argue as in \cite[Proposition 18]{MR}. The surface $\si$ separates $\Omega$ in two connected components $\Omega_1$ and $\Omega_2$, so is sufficient to construct a sweepout $\{\Sigma_t^i\}_{t\in[0,1]}$ as in the statement on $\Omega_i$, $i=1,2$. In fact, defining $\Sigma_t=\Sigma_{-t}^1$ if $t\le0$ and $\Sigma_t=\Sigma_t^2$ if $t\ge0$, the sweepout $\{\Sigma_t\}_{t\in[-1,1]}$ satisfies the properties stated.

Since $\Sigma$ is unstable, the first eigenvalue $\lambda_1$ of the Jacobi operator is negative. Also, we can choose a first eigenfunction $u_1$ associated to $\lambda_1$ to be positive. Let $N_i$ be the unit normal along $\si$ which points towards $\interior\Omega_i$. For $\varepsilon>0$ small enough,
the map $\Phi: \Sigma\times[0,\varepsilon]\to \Omega_i, \Phi(p,t) = \exp_p(tu_1(p)N_i(p))$ is well defined.

We then define $\Sigma_t^i=\Phi(\Sigma,t)$ and $\Omega_i^t = \Omega_i\setminus \Phi(\Sigma\times[0,t))$. Choose 
$\varepsilon>0$ arbitrarily small, so that $\{\Sigma_t^i\}_{t\in[0,\varepsilon]}$ defines a foliation of a neighborhood of $\Sigma$ whose leaves $\Sigma_t^i$ have non
vanishing mean curvature vector pointing towards $\Omega_i^t$.
Thus $|\Sigma_t^1|$ is decreasing for \( t \in [0, \epsilon] \) and, therefore, $|\Sigma_\varepsilon^i| < |\Sigma|-\delta$ for some $\delta>0$. Now in
order to construct the sweepout announced in the Step 2, it is sufficient to
construct a sweepout $\{\Sigma_t^i\}_{t\in[\varepsilon,1]}$ of $\Omega_i^\varepsilon$ such that
$\mathbf{L}(\{\Sigma_t^i\}_{t\in[\varepsilon,1]})\le |\Sigma|-\delta/ 2$. Indeed, we can glue such a
sweepout with the foliation $\{\Sigma_t^i\}_{t\in[0,\varepsilon]}$ to produce the sweepout of $\Omega_i$.

So let us assume by contradiction that any continuous sweepout $\{\Sigma_t^i\}_{t\in[\varepsilon,1]}$ of
$\Omega_{\varepsilon}^i$ satisfies \( \mathbf{L}\left(\{\Sigma_t^i\}_{t\in[\varepsilon,1]}\right) > |\Sigma|-\delta/2 > |\Sigma_{\varepsilon}^i| + \delta/2 \).
Let $\Pi$ be the smallest saturated set  containing \( \{\Sigma_t^i\}_{t \in [\varepsilon,1]} \). Since $\partial\Omega_i\setminus \si$ is strictly stable, we have $W(\Omega_i^\varepsilon,\Pi) > |\partial\Omega_i\setminus\si|$ (as proved in Appendix of \cite{KLS}).
Then the min-max Theorem \ref{smoothminmax}, together with Step $1$, implies that there is an unstable minimal surface $S$ in $\interior\Omega_i^{\varepsilon}$. 

By Theorem \ref{topology}, $\Omega$ is simply-connected, so the same holds for $\Omega_i^{\varepsilon}$. Then, the surface $S$ is orientable and separates $\Omega_i^{\varepsilon}$. If $\Gamma = \partial\Omega_i\setminus\si$ is non-empty, reasoning as in Step 2 of Theorem \ref{topology}, we conclude that $S$ is connected and homologous to $\Gamma$, and hence $S$ is homologous to $\si$. Now, suppose $\Gamma$ is empty. We know that $\mathrm{int}\,\Omega$ contains no embedded closed minimal surface whose orientable 2-cover is stable (by Step 1), then any two closed minimal surfaces in $\mathrm{int}\,\Omega$ have to intersect (see \cite{MPR}, Theorem $9.1$). Thus the min-max Theorem \ref{smoothminmax} implies that the sphere $S$ is connected. Moreover, $\Omega_i$ is a compact manifold with connected boundary, so by Theorem \ref{topology} it is diffeomorphic to a $3$-ball. Hence $S$ is homologous to $\si$.

Let $\widetilde{\Omega}_i$ be component of $\overline{\Omega_i\setminus S}$ which contains $\si$. Then, we can minimize area on the homology class of $S$ inside $\tilde\Omega_i$, and produce a minimal surface $S'$ on $\interior\widetilde\Omega_i \subset \interior\Omega_i$ whose orientable 2-cover is stable. However, this leads to a contradiction with Step 1.

Thus, we have proved that any minimal surface $S$ in $\Omega_i^\varepsilon$ produced by the Theorem \ref{smoothminmax} leads to a contradiction. Therefore there is a sweepout as in the statement of the claim.\\

\noindent
{\bf Step 3: $\si$ realizes the width and has index one.}\\

Denote by $\Pi$ the smallest saturated set containing the sweepout of Step 2, and let $W(\Pi)$ be the associated width. By Theorem \ref{smoothminmax} there exist disjoint closed embedded minimal spheres $\Sigma^\infty_i\subset (\Omega,g)$, $i= 1,\ldots,k$, and positive integers $n_1,\cdots,n_k$, such that $W(\Pi) = \sum_{i=1}^{k}n_i|\Sigma^{\infty}_{i}|$. Since $\interior(\Omega)$ has no embedded closed minimal surface whose orientable 2-cover is stable, any two closed minimal surfaces in $\interior(\Omega)$ have to intersect (see \cite{MPR}, Theorem $9.1$). Thus, exactly one of the surfaces $\si^{\infty}_{i}$, let us say $\si^{\infty}_{1}$, is contained in $\Omega$, and $\Sigma^{\infty}_{1}\cap\si \neq \emptyset$.

By Step $2$ the sweepout $\{\si_t\}_{t \in [-1,1]}$ satisfies ${\bf L}\left(\{\si_t\}\right) = |\si|$. Thus 
\begin{equation}\label{WLB2}
W(\Pi) \leq |\si|. 
\end{equation}
Moreover
\begin{equation}\label{WLB}
W(\Pi) = \sum_{i=1}^{k}n_i|\Sigma^{\infty}_{i}| \geq |\Sigma^{\infty}_{1}|.
\end{equation}
If $\si = \si^{\infty}_{1}$, then by \eqref{WLB2} and \eqref{WLB}, we have $W(\Pi) = |\si|$. Hence the sweepout $\{\si_t\}_{t \in [-1,1]}$ realizes $W(\Pi)$. So, it is an optimal sweepout and $\si$ is a min-max surface.
 
Thus, suppose $\si \neq \si^{\infty}_{1}$. We claim that $|\si| \leq |\si^{\infty}_{1}|$. The proof of this fact is divided in two cases:

\begin{description}
\item[Case 1] $\si$ and $\si^{\infty}_{1}$ intersect transversely.\\

In this case $\si\cap\si^{\infty}_{1}$ consists of a finite number $m$ of embedded smooth closed curves, which are pairwise disjoint. So we are left with the following sub-cases.\\

\noindent
(1a) $m = 1$\\

On this case, $\Sigma \setminus \si^\infty = S_1\cup S_2$, $\Sigma^\infty \setminus \si = \Gamma_1\cup \Gamma_2$, and $\partial S_i = \partial \Gamma_j = \gamma$, for $i,j \in \{1,2\}$. We can assume (without loss of generality) that $|S_1| \leq |S_2|$ and $\Gamma_1 \subset \Omega_1$. 

Suppose $|\Gamma_1| < |S_1|$. By \cite[Theorems 1 and 2]{MeeYau} there exists an embedded disk $D$ which minimizes area among disks contained in $\Omega_1$ and whose boundary is $\gamma$, moreover, either $D \subset \si$ or $D\cap\si = \gamma$. By the assumption, necessarily the second case holds. Denote $D_{t} = \Phi(D,t)$, where $\Phi$ is the flow defined by \eqref{flow}. Using Proposition \ref{goodlemma} and arguing as in the third paragraph of Proposition \ref{injective} we conclude that for all $t$ such that the flow is defined, $D_{t}$ is isometric to the hemisphere $\sw_+$ of constant curvature $\frac{\Lambda}{3}$. Arguing as in the last paragraph of Proposition \ref{injective} we conclude that the map $\Phi$ is defined for all time $t$, $\Omega_1$ is isometric to the canonical hemisphere $\st_+$ of constant curvature $\frac{\Lambda}{3}$, and when $t \to \pm\infty$ the surfaces $D_{t}$ converge to either $S_1$ or $S_2$. Thus, $S_1$, $S_2$ and $D$ are all isometric, which contradicts the inequality $|D| \leq |\Gamma_1| < |S_1|$. Thus $|\Gamma_1| \geq |S_1|$.

Suppose now, $|\Gamma_1| < |S_2|$. Let $U$ be the component of $\Omega_1\setminus\Gamma_1$ which contains $S_2$. Then $\overline U$ is a mean convex domain with piece-wise smooth boundary. Using again \cite[Theorems 1 and 2]{MeeYau} we find an embedded disk $\tilde D$ which minimizes area among disks contained in $\overline U$ and whose boundary is $\gamma$, and by the assumption \( \tilde D \not\subset \si \), necessarily $\tilde D\cap\si = \gamma$. Arguing as in the last paragraph we conclude that $\Omega_1$ is isometric to the canonical hemisphere $\st_+$ of constant curvature, and $S_2$, $\Gamma_1$ and $\tilde D$ are isometric to half-equators on this hemisphere, which is again a contradiction. Hence $|\Gamma_1| \geq |S_2|$.

We can argue similarly to prove that $|\Gamma_2| \geq \max\{|S_1|,|S_2|\}$. Therefore 
$$|\si| = |S_1| + |S_2| \leq |\Gamma_1| + |\Gamma_2| = |\si^{\infty}_{1}|.$$

\noindent
(1b) $m \geq 2$\\

Let $S$ be a connected component of $\Sigma\setminus\si^\infty$. Since $m \geq 2$ and $\si$, $\si^{\infty}_{1}$ are embedded spheres, necessarily $\Sigma^\infty\setminus\si$ has exactly one connected component (which we denote by $\Gamma$) whose boundary is $\partial S$. We can assume (without loss of generality) that $\Gamma \subset \Omega_1$. 

Suppose that $|\Gamma| < |S|$. Let $U$ be the oriented surface which minimizes area among surfaces contained in $\Omega_1$ and whose boundary is $\partial S$ (see  Appendix \ref{appendix}). By our hypothesis, necessarily $U\cap\si = \partial S$. Consider a connected component $D$ of $U$. We have,
\begin{equation}\label{eq.case1b}
|D| \leq |U| \leq |\Gamma| < |S|.
\end{equation}
On the other hand, we can proceed as in the second paragraph of case (1a) and conclude that $D$ is isometric to is isometric to the hemisphere $\sw_+(\frac{\Lambda}{3})$, and $\Omega_1$ is isometric to the hemisphere $\st_+(\frac{\Lambda}{3})$. Thus, $\partial D$ separates $\si$ in two connected components $D_1,D_2$, each one isometric to $\sw_+(\frac{\Lambda}{3})$, and $S \subset D_2$. So,
$$|S| \leq |D_2| = |D|,$$
which contradicts \eqref{eq.case1b}. Therefore, $|S| \leq |\Gamma|$.

Since $S$ was an arbitrary connected component of $\Sigma\setminus\si^\infty$, we have $|\Sigma|\leq|\si^\infty|$.
\\

\item[Case 2] $\si$ and $\si^{\infty}_{1}$ intersect tangentially at some
point.\\

In this case, since we are supposing $\si \neq \si^{\infty}_{1}$, the Maximum Principle implies that we can not have one surface contained on one side of the other. So, by Lemma 1.4 of \cite{FHS}, if $p\in \Sigma\cap \Sigma^{\infty}$ is a point where the surfaces are tangent, then there exists a neighborhood $\mathcal{U}$ of $p$ such that $\Sigma\cap\Sigma^{\infty}\cap\mathcal{U}$ is given by $2k$ arcs, $k\geq 2$, starting at $p$ and making equal angle. We call such a point $p$ a $k$-prong singularity. It follows from this description that any connected component $S$ of $\Sigma\setminus\si^\infty$ (or $\Sigma^\infty\setminus\si$) whose closure $\bar{S}$ contains $k$-prong singularities, is such that $\bar{S}$ is an almost properly embedded surface (in the sense of definition \ref{def:alm.prop.emb}).

Let $S$ be a connected component of $\Sigma\setminus\si^\infty$, such that $\overline{S}$ contains a point of tangency. By the structure of $\Sigma\cap\si^\infty$ and the fact that $\Sigma$, $\si^\infty$ are embedded spheres, it follows that there is exactly one connected component (which we denote by $\Gamma$) of $\Sigma^\infty\setminus\si$ whose boundary is $\partial S$. We want to prove that $|S| \leq |\Gamma|$. On cases (1a) and (1b) we faced similar situations. There, we supposed $|\Gamma| < |S|$ and considered the solution of the area minimization problem with fixed boundary $\partial S$ on a suitable domain. In the case, $\partial S$ has no singularities, the solution of this problem is a smooth surface with boundary, so we can apply  Proposition \ref{goodlemma} to obtain a contradiction. We want to use the same strategy here. The main difficulty now is that a regularity theorem near the singularities of $\partial S$ it is not available in the literature.

However, we claim that is still possible to argue as in cases (1a) and (1b). Assume (without loss of generality) that $\Gamma \subset \Omega_1$. Suppose, $|\Gamma| < |S|$. Let $T$ be a solution of the following problem
$${\bf M}(T) = \inf\{{\bf M}(\tilde{T}); \ \tilde{T} \in \mathfrak{C}\},$$
where
\begin{eqnarray*}
\mathfrak{C} = \{\tilde T \in {\bf I}_2(M); \ \spt \tilde T \subset \Omega_1 \ \mbox{and} \ \spt \bigl(\partial \tilde T\bigr) = \partial S\}.
\end{eqnarray*}
Let $p_1,\cdots,p_n$ be the singularities of $\partial S$. The results in Appendix \ref{appendix} imply that $(\spt T)\setminus(\partial T)$ is a smooth embedded surface and for any $p \in \partial S\setminus\{p_1,\cdots,p_n\}$ there is a neighbourhood $W$ of $p$ such $\spt T\cap W$ is a smooth embedded surface with boundary. Let $D$ be a connected component $\spt T\setminus\{p_1,\cdots,p_n\}$ such that $\partial \bar{D}$ is not empty ($\bar{D}$ is the topological closure of $D$). Then $\bar{D}$ is an almost properly embedded surface (in the sense of Definition \ref{def:alm.prop.emb}) whose boundary is $\partial S$, so  Lemma \ref{goodlemma} is applicable for it. Arguing as before, we conclude that $\bar D$ is isometric to a domain $U$ bounded by geodesics in $\sw(\frac{\Lambda}{3})$, the region bounded by $S$ and $\bar D$ is isometric to a region in $\st(\frac{\Lambda}{3})$, and $S$ is contained in a domain of $\si$ which is isometric to $U$. Thus,
$$|S| \leq |U| \leq {\bf M}(T) \leq |\Gamma|,$$
which is a contradiction with our assumption. Therefore $|S| \leq |\Gamma|$. 

Combining this we the previous case we conclude $|\si| \leq |\si^{\infty}_{1}|$.
\newline
\end{description}

Therefore, $|\si| = |\si^{\infty}_{1}| = W(\Pi)$, and as in the case $\si = \si^{\infty}_{1}$ we conclude that $\si$ is a min-max surface. By the index estimates of Marques and Neves \cite{MaNeindexbound}, it follows that $\si$ has index at most one. Since $\si$ is unstable, it must have index equal to one.
\end{proof}

%\begin{rem} There is an alternative version of Min-Max theory for minimal surfaces, due to Almgreen-Pitts, where there is no continuous dependence of the leaves of the sweepout with respect to the parameter, and there is no control on the topology of the leaves. In general, for manifolds which admit sweepouts by spheres, it is not known if the Almgreen-Pitts width (denoted $W_{AP}$) coincides with the continuous width of spheres. Actually, there are examples where the two widths are in fact different (e.g. in some class of Berger Spheres). In our case, with more effort it is possible to prove that the $\Sigma$ also realizes $W_{AP}$. As above we can prove that the Almgreen-Pitts width satisfies $W_{AP} = |\Sigma|$. Finally, arguing along the same lines of the end of the proof of proposition 20 in \cite{MR} we conclude that $\si$ is a min-max surface. \end{rem}

\begin{theo}\label{char2}
Consider a complete electrostatic system $(M^3,g,V,E)$, such that $\Lambda + |E|^2 > 0$. Let $\Omega_1, \Omega_2$ be connected maximal regions where $V \neq 0$, such that $\overline\Omega_i$ is compact and let $\si = \partial\Omega_1\cap\partial\Omega_2 \subset V^{-1}(0)$ be unstable. Suppose $\partial(\Omega_1\cup\Omega_2)$ is non empty and stable, and at least one of its components is degenerate. Then $\si$ has index one.
\end{theo}

\begin{proof}
Denote $\Omega = \Omega_1\cup\Omega_2$. The idea of the proof is the following. We consider a certain sequence of metrics $\{g_i\}$ converging to $g$ and such that $g_i = g$ in a neighborhood of $\si$. Then we will prove that for each $i$, $\si$ is a minimal surface of index 1 in $(M,g_i)$ which can be obtained by min-max methods. \\

\noindent
{\bf Step 1: Construction of the sequence of metrics.}\\ 

Define $T_{r}(\partial \Omega) = \big\{x \in M;\, \mathrm{dist}_g(x,\partial \Omega) \leq r\big\}$. Choose $\delta > 0$ sufficiently small so that the function $x \mapsto \mathrm{dist}_g(x,\partial \Omega)^2$ is smooth in $T_{3\delta}(\partial \Omega)$, and such that $\si\cap T_{3\delta}(\partial \Omega) = \emptyset$. Let $\eta$ be a smooth function such that $\eta \equiv 1$ in $T_{\delta}(\partial \Omega)$ and $\eta \equiv 0$ in $M\setminus T_{2\delta}(\partial \Omega)$, with $0\leq \eta \leq 1$ in $\Omega$. Now, define $h(x) = \eta(x)\,\mathrm{dist}_g(x,\partial \Omega)^2$ for $x \in T_{3\delta}(\partial \Omega)$, and $h(x) = 0$ for $x \in M\setminus T_{3\delta}(\partial \Omega)$. Then $h: M \to \mathbb{R}$ is a smooth function which coincides with $\mathrm{dist}(\cdot,\partial \Omega)^2$ in a neighborhood of $\partial \Omega$.

Define the sequence of metrics $g_i = e^{\frac{2}{i}h}g$, where $i$ is a positive integer. Observe that for any $i$ we have $g_i = g$ in $M\setminus T_{3\delta}(\partial \Omega)$, so $\si$ is still minimal with respect to $g_i$. Also, as proved in \cite[Proposition 2.3]{IMN}, $\partial\Omega$ is still minimal with respect to $g_i$, and the spectrum of its Jacobi operator satisfies
$$\mathrm{spec}\big(L_{\partial\Omega,g_i}\big) = \mathrm{spec}\big(L_{\partial\Omega,g}\big) + \frac{4}{i}.$$
Thus, each component of $\partial\Omega$ is strictly stable in $(M,g_i)$, $\forall\,i$.\\

\noindent
{\bf Step 2: There is a foliation of a neighborhood of $\partial\Omega$ such that, with respect to any of the metrics $g_i$, the leaves have non-vanishing mean curvature vector pointing towards $\partial\Omega$.}\\

Write $\partial\Omega = \cup_{j=1}^{\ell}\partial_j \Omega$. Arguing as in Proposition 3.2 of \cite{BBN} it follows that there is there exists $\epsilon > 0$ and $\Psi: \partial\Omega\times [0,\epsilon]\to (\Omega,g)$ which is a diffeomorphism over its image, such that $\forall\, j$ the following holds:
\begin{enumerate} 
\item $\Psi|_{\partial_j\Omega\times\{0\}}$ is the identity;
\item $S_{j,t} = \Psi(\partial_j\Omega\times\{t\})$ is either a minimal surface or has non-vanishing mean curvature.
\end{enumerate}
We denote by $H_{j,t,g}$the mean curvature of $S_{j,t}$ with respect to the unit normal $N_{j,t}$ which points away from $\partial\Omega$.

Suppose there exist $0 < t_1 < t_2$ such that $S_{j,t_1}$ and $S_{j,t_1}$ are minimal surfaces. Minimizing area in the homology class of $S_{j,t_1}$ inside the region between $S_{j,t_1}$ and $S_{j,t_2}$, we obtain an embedded closed minimal surface in $(\mathrm{int}\,\Omega,g)$ whose orientable 2-cover is stable. However, as in Step 1 of the previous theorem we can prove that $(\mathrm{int}\,\Omega,g)$ contains no embedded minimal surface whose orientable 2-cover is stable. So we have a contradiction. Thus, decreasing $\epsilon$ if necessary, the mean curvature satisfies 
$$H_{j,t,g}(p) \neq 0,\ \forall\, p\in S_{j,t},\ \forall\, 0 < t \leq \epsilon.$$
This implies $\partial_j\Omega$ minimizes area in its homology class. Hence, by the first variation formula, we have $H_{j,t,g} > 0$ necessarily.

Consider $\epsilon_0$ such that $S_{j,t} \subset T_{\delta}(\partial\Omega), \forall\, j, \forall\, t \in (0,\epsilon_0)$. In $T_{\delta}(\partial\Omega)$ it holds $\nabla_g h = 2\,\mathrm{dist}_g(\cdot,\partial\Omega)\nabla_{g}\big(\mathrm{dist}_g(\cdot,\partial\Omega)\big)$. Thus, since $\nabla_{g}\big(\mathrm{dist}_g(\cdot,\partial\Omega)\big)$ points away from $\partial\Omega$, we have $\langle\nabla_g\,h,N_{j,t}\rangle \geq 0$. Finally, the formula for the mean curvature of $S_{j,t}$ with respect to the metric $g_i$ give us
\begin{align*}
H_{j,t,g_i} &= e^{-\frac{1}{i}h}\left(H_{j,t,g} + \frac{1}{i}\langle\nabla_g h,N_{j,t}\rangle\right) > 0.
\newline
\end{align*}

\noindent
{\bf Step 3: There is an infinite set $\mathbb{N}^{\prime} \subset \mathbb{N}$, such that for any $i \in \mathbb{N}^{\prime}$, the Riemannian manifold $(\mathrm{int}\,\Omega,g_{i})$ does not contain embedded minimal surfaces whose orientable 2-cover is stable, and whose area is less than or equal to $|\si|_{g}$.}\\ 

Suppose that for any subsequence $\{g_{i_k}\}$ there is $\Sigma_{i_k}$ a connected stable embedded minimal surface in $(\mathrm{int}\,\Omega,g_{i_k})$ with area smaller than $|\si|_g$, for each $i_k$. Arguing as in Proposition B.1 in \cite{MaNe} we conclude that a subsequence of $\{\si_{i_k}\}$ converges to a closed stable minimal surface $\widetilde{\si}$ in $(\Omega,g)$ whose area is less than or equal to $|\si|_{g}$. We claim that $\widetilde{\si} \subset \mathrm{int}\,\Omega$. By the previous step and the maximum principle the region between $\partial\Omega$ and $\cup_{j=1}^{\ell}S_{j,\epsilon_0}$ contains no minimal surface other than $\partial\Omega$. So, for any $i_k$ we must have $\si_{i_k}\cap\big(\Omega\setminus(\cup_{j=1}^{\ell}\cup_{t=0}^{\epsilon_0}S_{j,t})\big) \neq \emptyset$. So, $\widetilde{\si}\cap\mathrm{int}\,\Omega \neq \emptyset$, and the claim follows by the maximum principle. 

However, the existence of $\widetilde{\si}$ contradicts that $(\mathrm{int}\,\Omega,g)$ contains no embedded minimal surface whose orientable 2-cover is stable. \\ 

\noindent
{\bf Step 4: For any $i \in \mathbb{N}^{\prime}$, there is a sweepout $\{\Sigma_t\}_{t\in[-1,1]}$ of $(\Omega,g_i)$ such that
$\mathbf{L}(\{\Sigma_t\})=|\Sigma|$, $\si_{-1} = \partial\Omega_1\setminus\si$, $\si_1 = \partial\Omega_2\setminus\si$, $\si_0 = \si$, and for any
$\varepsilon>0$, there is $\delta_0>0$ such that $|\Sigma_t|_{g_i} \leq |\Sigma|_{g} -\delta_0$, if $|t| \ge \varepsilon$.}\\

Observe that for any embedded surface $S \subset \Omega$  and $k > i$ we have
\begin{equation}\label{eq:ineq.metrics}
|S|_{g_i} = \int_{S} e^{\frac{2}{i}h}\,d\mu_g \geq \int_{S} e^{\frac{2}{k}h}\,d\mu_g = |S|_{g_k}.
\end{equation}
So we will construct the sweepout for $(\Omega,g_1)$, and the same family will have the required properties with respect to any of the metrics $g_i$.

Now, the proof follows exactly as in Step 2 of the previous theorem. The only difference being the following one. Recall that in the construction we obtain a region inside $\Omega$ whose boundary consists of $\si$ and a unstable minimal surface $S$. Then we minimize area in the homology class of $\si$ to obtain an embedded minimal surface in $\mathrm{int}\,\Omega$ whose orientable 2-cover is stable. The contradiction on the previous case follows from the fact $(\mathrm{int}\,\Omega,g)$ admits no such surface. In the present situation the surface obtained by minimization necessarily has area less than or equal to $|\si|_{g}$ and this contradicts the Step 2 of the current proof.\\ 

\noindent
{\bf Step 5: There is an embedded minimal sphere $\si^{\infty}$ in $(\mathrm{int}\,\Omega,g)$ whose index is equal to one and the area is equal to $|\si|_{g}$. This implies $\si$ has index 1 with respect to $g$}.\\ 

Consider the subsequence $\{g_{i}\}_{i \in \mathbb{N}^{\prime}}$ of step 3. Let $\Pi$ be the smallest saturated set containing the sweepout of step 4 (which is the same for all $i$), and let $W_{i}$ be the associated width in $(\Omega,g_i)$. Applying Theorem \ref{smoothminmax} and the index estimates \cite{MaNeindexbound}, we obtain an embedded minimal sphere $\si_{i}^{\infty}$ in $(\mathrm{int}\,\Omega,g_i)$ (one of the components of the min-max surface) of index at most one. Moreover, by construction
$$|\si_{i}^{\infty}|_{g_i} \leq W_{i} \leq |\si|_g,$$
and so by step $3$, $\si_{i}^{\infty}$ has index one.

Now we want to conclude that $W_{i} = |\si|_g = |\si|_{g_i}$, so that $\si$ realizes the maximum of the area in an optimal sweepout and hence it is a min-max surface which realizes the width $W_{i}$. This implies that $\si$ has index one with respect to all the metrics $g_i$, and by the lower semi-continuity of the index the same holds for the metric $g$. We could try the same strategy as in the proof of the previous theorem, comparing $|\si_{i}^{\infty}|_{g_i}$ with $|\si|_{g}$, however on that case we have used that there was $(V,E)$ solving the electrostatic equations with respect to the metric $g$, and this is not necessarily true for $g_i$. So we will apply that argument to a limit of $\{\si_{i}^{\infty}\}_{i}$ on $(M,g)$.

By \cite[Theorem 2.1]{W2}, a subsequence of $\{\si_{i}^{\infty}\}_{i}$ converges to an embedded minimal surface $\si^{\infty}$ in $(\Omega,g)$ whose index is at most one and whose area is less than or equal to $|\si|_g$. Arguing as in Step 3 we conclude that $\si^{\infty} \subset \mathrm{int}\,\Omega$, hence the index of $\si^{\infty}$ is one. Also, by the characterization in \cite{W2}, the multiplicity in the convergence is one, otherwise $\si^{\infty}$ would be non orientable or stable orientable, and $(\Omega,g)$ does not contain such type of surfaces. Since, $\si_{i}^{\infty}$ is a sphere for all $i$, this implies $\si^{\infty}$ is also a sphere. 

Now, we can proceed as in the step 3 of the previous theorem to prove that $|\si^{\infty}|_g = |\si|_{g}$. Since $|\si_{i}^{\infty}|_{g_{i}} \to |\si|_g$ and $|\si_{i}^{\infty}|_{g_{i}} \leq W_{i} \leq |\si|_g$, it follows that $ W_{i}$ converges to $|\si|_g$. Suppose that $W_{i} < |\si|_{g}$ for some $i$. By equation \eqref{eq:ineq.metrics} and the fact the class of sweepouts is the same we have 
$$W_{k} \leq W_{i} < |\si|_{g},\ \forall\, k > i,$$ 
which contradicts the fact that the sequence of widths converges to $|\si|_g$. Thus $W_{i} = |\si|_{g}$ for any $i \in \mathbb{N}^{\prime}$. 
\end{proof}

\begin{rem}
The only known examples of electrostatic systems such that $V^{-1}(0)$ contains a stable degenerate component are the ones where $(M,g)$ is isometric to a standard cylinder $(I\times\mathbb{S}^2,ds^2 + g_{\mathbb{S}^2})$. In these cases there is no unstable component in $V^{-1}(0)$. So, it is an open question if the situation in the previous theorem can actually happen.
\end{rem}

\section{Area-Charge Inequality and Rigidity}\label{Area-Charge ineq}

\subsection{Area-Charge Inequality}

Consider a time-oriented spacetime $(\mathcal{M}^4,\mathfrak{g})$, and a 2-form $F$ on $\mathcal{M}$. The Einstein-Maxwell equations for uncharged matter are
\begin{eqnarray*}
  \Ric_{\mathfrak{g}}-\frac{R_{\mathfrak{g}}}{2} \mathfrak{g} + \Lambda \mathfrak{g} = 8\pi (T_F + T),\\
  dF=0, \quad \diver_{\mathfrak{g}} F=0,
\end{eqnarray*}
where $T_F$ is the energy-moment tensor of the electromagnetic field, defined in \eqref{energ.mom}, and $T$ is a symmetric covariant $2$-tensor field on $\mathcal{M}$, which represents the energy-moment tensor of the non-electromagnetic matter.
 
A \emph{Cauchy data} for these equations (without a magnetic field) is a tuple 
$$(M^3,g,K,E,\mu,J),$$ where $(M^3,g)$ is an oriented Riemannian 3-manifold such that $M \subset \mathcal{M}$ and $g = \mathfrak{g}|_M$, $K$ is the second fundamental form of $M$ inside $\mathcal{M}$ with respect to the future directed unit timelike normal $n$, $E \in \mathfrak{X}(M)$ is the electric field, and $\mu := T(n,n)$, $J := T(\cdot,n)$ are respectively the energy density and the momentum density of the non electromagnetic fields. Moreover, the data satisfy the constraint equations
\begin{align*}
R_g + \bigl(\tr_g K\bigr)^2 - |K|_{g}^2 - 2\Lambda &= 2|E|^2 + 16\pi\mu,\\
\diver_g\bigl(K - (\tr_g K)g\bigr) &= 8\pi J,\\
\diver_g E &= 0.
\end{align*} 

In this setting, we say it holds the Dominant Energy Condition for the non electromagnetic fields if (\cite[Section 2]{DG})
$$T(v,v) \geq 0,$$ 
for all future directed causal vectors. This condition implies $\mu \geq |J|_{g}$ (\cite[Section 8.3.4]{GourgoulhonFormalismGeneralRelativity2012}), so assuming $(M,g,K)$ is \emph{maximal}, i.e., $\tr_g K =0$, we obtain
$$R_g \geq 2\Lambda + 2|E|^2.$$

Let $\Sigma \subset M$ be a closed orientable embedded surface with a unit normal $N$. Recall the expression \eqref{defcharge} for the charge $Q(\si)$. Since $\diver_g E=0$, it follows from the Divergence Theorem that the charge depends only on the homology class of $\si$. This context motivates the assumptions of the following result.

\begin{prop}\label{mainineq}
Let $(M^3,g)$ be an oriented Riemannian manifold and $E \in \mathfrak{X}(M)$ such that $R_{g} \geq 2\Lambda + 2|E|^2$, for some $\Lambda \in \re$. Suppose $\si \subset (M,g)$ is an orientable closed minimal surface of index one, with unit normal $N$. Then,
\begin{equation}\label{areacharge}
\Lambda|\si| + \frac{16\pi^2 Q(\si)^2}{|\si|} \leq 12\pi + 8\pi\Biggl(\frac{g(\si)}{2} - \biggl[\frac{g(\si)}{2}\biggr]\Biggr).
\end{equation}
In particular, if $\Lambda > 0$ and $g(\si) = 0$ we have
\begin{eqnarray}
Q(\si)^2 &\leq& \frac{9}{4\Lambda}\label{chargebound}\;,\\
\frac{2\pi}{\Lambda}\Bigl(3 - \sqrt{9 - 4\Lambda Q(\si)^2}\Bigr) \leq &|\si|& \leq \frac{2\pi}{\Lambda}\Bigl(3 + \sqrt{9 - 4\Lambda Q(\si)^2}\Bigr)\label{areabound}.
\end{eqnarray}

\noindent
Moreover, the equality in \eqref{areacharge} holds if, and only if, $\si$ is totally geodesic, $E|_{\si} = aN$, for some constant $a$, $(R_g)|_{\si} \equiv 2\Lambda + 2a^2$, and $g(\si)$ is an even integer.
\end{prop}

\begin{proof}
The proof is based on the Hersch's trick. Let $u_1$ be
the first eigenfunction of the Jacobi operator of $\Sigma$, which we can choose to be positive. Let $\phi$ be a conformal map
from $\Sigma$ to $\mathbb{S}^2 \subset \rt$ and consider the following integral
$$\int_\Sigma  (h\circ \phi)u_1 d\mu\in \rt,$$
where $h$ is a M\"obius tranformation of $\mathbb{S}^2$. Since $u_1$ is positive, we can find $h$ such that the above integral vanishes (see \cite{LiYa}). 
Also, by \cite{YY}, we have 
$$\int_\Sigma |\nabla_{\si} (h\circ\phi)|^2d\mu = 8\pi\deg(h\circ\phi).$$
As in \cite{RR}, we can choose $\phi$ such that 
$$\deg(\phi) \leq 1+\left[\frac{g(\Sigma)+1}{2}\right].$$

Let $(f_1,f_2,f_3)$ be the three
coordinates of $h\circ \phi$. Then $f_i$ is orthogonal to $u_1$. Since $\Sigma$ has index one we have
$$\int_\Sigma \bigl(|\nabla_{\si} f_i|^2-(\Ric_g(N,N) + |A|^2)f_i^2 \bigr)d\mu\ge 0, \ i=1,2,3.$$
Summing these three inequalities, we get

\begin{align*}
0 &\leq \int_\Sigma \bigl(|\nabla_{\si} (h\circ\phi)|^2-(\Ric_g(N,N)+|A|^2)\bigr)\,d\mu\\
&=8\pi\deg(h\circ\phi) + \int_\Sigma K_{\si}d\mu - \frac{1}{2}\int_\Sigma (R_g + |A|^2)\,d\mu\\
&\leq 8\pi\deg(\phi) + 4\pi\left( 1 - g\left( \si \right) \right) - \int_\Sigma (\Lambda + |E|^2)\,d\mu\\
&\leq 8\pi\Biggl(1+\left[\frac{g(\Sigma) + 1}{2}\right]\Biggr) + 4\pi\left( 1 - g\left( \si \right) \right) - \Lambda|\si| - \int_\Sigma \langle E, N\rangle^2\,d\mu\\
&\leq 12\pi + 8\pi\Biggl(\frac{g(\si)}{2} - \biggl[\frac{g(\si)}{2}\biggr]\Biggr) - \Lambda|\si| - \frac{1}{|\si|}\biggl(\int_\Sigma \langle E, N\rangle\,d\mu\biggr)^2\\
&= 12\pi + 8\pi\Biggl(\frac{g(\si)}{2} - \biggl[\frac{g(\si)}{2}\biggr]\Biggr) - \Lambda|\si| - \frac{16\pi^2 Q(\si)^2}{|\si|},
\end{align*}
where we have used the Gauss equation in the second line, the Gauss-Bonnet Theorem and the Dominant Energy Condition in the third line, the Cauchy-Schwarz inequality in the fourth line and the H\"older's inequality in the fifth line.

Suppose $g(\si) = 0$. Then $|\si|$ is a real number satisfying the quadratic inequality $\Lambda x^2 - 12\pi x + 16\pi^2 Q(\si)^2 \leq 0$. If $\Lambda > 0$ this is only possible if the discriminant of the quadratic expression is non-negative, thus we obtain
$$144\pi^2 - 64\pi^2\Lambda Q(\si)^2 \geq 0,$$
which implies \eqref{chargebound}. Solving the quadratic inequality we obtain \eqref{areabound}.

If equality holds in \eqref{areacharge}, then we have equality in all the previous inequalities. So, $A|_{\si} \equiv 0$, and the Jacobi operator is equal to $\Delta_{\si} + \Ric_g(N,N)$. The equality on Cauchy-Schwarz inequality implies $|E|^2 = \langle E, N\rangle^2$ and $E = aN$ for some function $a = \langle E, N\rangle$. On the other hand, the equality on H\"older's inequality gives us that $\langle E, N\rangle$ is constant. Thus, 
\begin{equation*}\label{scal}
(R_g)|_{\si} \equiv 2\Lambda + 2a^2.
\end{equation*}

Moreover, each $f_i$ is on the kernel of the Jacobi operator, which implies $\Delta_{\si} (h\circ\phi) + \Ric_g(N,N)(h\circ\phi) = 0$. As any meromorphic map is harmonic, we also have $\Delta_{\si} (h\circ\phi) + |\nabla_{\si}(h\circ\phi) |^2(h\circ\phi)  = 0$. Therefore $|\nabla_{\si}(h\circ\phi) |^2 = \Ric_g(N,N)$, and the operator $\Delta_{\si} + |\nabla_{\si}(h\circ\phi)|^2$ has index one. By Theorem $4$ in \cite{Ro}, it follows that $\deg(h\circ\phi) \leq \bigl[\frac{g(\si)}{2}\bigr] + 1$. However, by assumption we also have 
$\deg(h\circ\phi) = 1+\left[\frac{g(\Sigma)+1}2\right].$ Thus 
$\left[\frac{g(\Sigma)+1}2\right] = \left[\frac{g(\Sigma)}2\right],$
and this is only possible if $g(\si)$ is even.
\end{proof}

\subsection{Rigidity of electrostatic systems}\label{Rigidity}

In Proposition \ref{mainineq} we obtained an inequality relating area and charge under a hypothesis of local nature on $\si$, and an infinitesimal rigidity for $\si$ under the assumption of equality in \eqref{areacharge}. The main goal of this subsection is to obtain global rigidity results for $(M,g)$ in the case of electrostatic systems. We first recall the following theorem.

\begin{theo}[Marques-Neves, \cite{MaNe}]\label{Marq.Nev}
Suppose $(\st,g)$ has no embedded stable minimal spheres and $R_g \geq 2\Lambda$, for some constant $\Lambda > 0$. Then there exists an embedded minimal sphere
$\Sigma$, of index one, such that
\begin{eqnarray*}\label{widthbound}
W(\st,g) = |\Sigma| \leq \frac{12\pi}{\Lambda}.
\end{eqnarray*}
The equality holds if, and only if, $g$ has constant sectional curvature $\frac{\Lambda}{3}$.
\end{theo}

We have then the following result.

\begin{theo}\label{rigstatic}
Let $(M^3,g,V,E)$ be an electrostatic system such that $M$ is closed, $\Lambda + |E|^2 > 0$ and $V^{-1}(0) = \si$ is non-empty and connected. Then $\si$ is a separating sphere, and
\begin{equation*}
\Lambda|\si|  \leq 12\pi,
\end{equation*}
with equality if, and only if, $E \equiv 0$ and $(M,g)$ is isometric to the standard sphere of constant sectional curvature $\frac{\Lambda}{3}$.
\end{theo}

\begin{proof}

Lets us first prove that $\si$ separates $M$. We claim $V$ changes sign on every neighbourhood of a point $p \in \si$. In fact, we saw that $|\nabla_g V|$ never vanishes on $\si$ (Lemma \ref{lemma:propert}), so a unit normal for $\si$ is given by $N = \frac{\nabla_g V}{|\nabla_g V|}$. Consider a geodesic $\gamma$ of $(M,g)$ such that $\gamma(0) = p$, and $\gamma'(0) = N(p)$, and the function defined by $f = V\circ \gamma$. We have $f(0) = 0$ and 
$$f'(0) = \langle\nabla_g V(p),N(p)\rangle = |\nabla_g V(p)| > 0.$$
Thus, $f$ it is strictly increasing in a sufficiently small neighbourhood of
$0$, which implies that for t sufficiently small we have $V\big(\gamma(t)\big) < V(p) = 0$ if $t < 0$ and $0 = V (p) < V\big(\gamma(t)\big)$ if $t > 0$. This proves the claim, and we conclude it holds the non-trivial decomposition
$$M\backslash\si = \{x \in M;\, V(x) < 0\}\cup\{x \in M;\, V(x) > 0\}.$$

If $\si$ is stable, choosing the function $f \equiv 1$ in the stability inequality and proceeding as in the proof of Proposition \ref{mainineq} we conclude that $\si$ is a $2$-sphere (using the Gauss-Bonnet Theorem) and obtain the inequality 
$$\Lambda|\si| \leq 4\pi < 12\pi.$$

Now, assume $\si$ is unstable. Let $M_1$ and $M_2$ be the connected components of $M\backslash\si$. By Theorem \ref{topology} we conclude that $(M_i,g)$ is simply connected, for $i=1,2$, and hence $\si$ is a $2$-sphere. Therefore $M$ is homeomorphic to $\st$. So $\Sigma$ is  null homologous, and hence $Q(\si) = 0$. Combining Theorems \ref{char1} and \ref{Marq.Nev} we have
\begin{eqnarray*}
W(\st,g) = |\Sigma| \ \ \textrm{and} \ \ \Lambda|\si| \leq 12\pi.
\end{eqnarray*}
Since \( R_g = 2\Lambda + 2|E|^2 \) (see Lemma \ref{lemma:propert}), it follows from Theorem \ref{Marq.Nev}
that the equality holds if, and only if $E \equiv 0$ and $g$ has constant sectional curvature $\frac{\Lambda}{3}$.
\end{proof}

Recall the generalization of Schoen's identity \cite{Sc} due to Gover and Orsted \cite[Sec.~3.1]{GO} for locally conserved tensor fields.

\begin{theo}[Gover-Orsted, \cite{GO}]
\label{psident}
Let $(M^n,g)$ be a compact Riemannian manifold with boundary $\partial M$. If $T$ is a symmetric $2$-tensor field such that $\diver_g T_g=0$, and $Z$ is a vector field on $M$, then
\begin{equation}\label{schoenpoh}
  \int_{M}Z(\tau)\,d\mathcal{V}=\frac{n}{2}\int_{M} \big\langle \mathring{T},\mathcal{L}_{Z}g\big\rangle\,d\mathcal{V} -n\int_{\partial M}\mathring{T}(Z,N)\,d\mu,
 \end{equation}
 where $\tau=\mathrm{tr}_{g}\,T$, $\mathring{T}$ denotes the trace free part of $T$ and $N$ denotes the outer unit normal along $\partial M$.
\end{theo}

As an application of \eqref{schoenpoh} we have.

\begin{prop}\label{prop-sphere}
Let $(M,g,V,E)$ be a compact electrostatic system, such that $V^{-1}(0) = \partial M = \displaystyle\cup_{i=1}^{\ell}\partial_i M$. Then,
\begin{eqnarray}
2\pi\sum_{i=1}^{\ell} k_i\,\chi(\partial_i M) &=& \int_{M}\biggl(|\mathrm{S}_g|^2 + \frac{2}{3}(\Lambda - |E|^2)|E|^2\biggr)V\,d\mathcal{V}\nonumber\\
&& + \ \sum_{i=1}^{\ell} \int_{\partial_i M}k_i\left(\frac{\Lambda}{3}+|E|^2\right)d\mu_i,
\end{eqnarray}
where $\mathrm{S}_g = \Ric_{g}-\frac{R_g}{3}g+2E^\flat\otimes E^\flat-\frac{2}{3}|E|^2g$, and $k_i = \big|(\nabla_g V)|_{\partial_i M}\big|$.
%Moreover, if $\sup_M |E|^2 \leq \Lambda$, then there is a boundary component of $\partial M$ which is diffeomorphic to a two-sphere.
\end{prop}

\begin{proof}
Without loss of generality we can suppose $V > 0$ on $\interior M$. Define $T_g = \Ric_{g}-\frac{R_g}{2}g+2E^\flat\otimes E^\flat-|E|^2g$. Then $\mathring{T}_g = \mathrm{S}_g$ and $\diver_g T_g=0$. Denote $Z=\nabla_g V.$ First, notice that 
\begin{eqnarray*}\label{primeq}
  \int_{M}Z(\tau)\,d\mathcal{V} &=&  \int_{M}\big\langle\nabla_g V,\nabla_g(-\Lambda-2|E|^2)\big\rangle\,d\mathcal{V}\\
  &=& 2 \int_{M}|E|^2\Delta_g V d\mathcal{V}+2 \int_{\partial M}|E|^2|\nabla_g V|\,d\mu\nonumber\\
    &=& 2 \int_{M}|E|^2(-\Lambda+|E|^2)V d\mathcal{V}+2 \int_{\partial M}|E|^2|\nabla_g V|\,d\mu,\nonumber
\end{eqnarray*}
where we have used \eqref{eq2} and the divergence Theorem. Here, the unit normal is $N =-{\nabla_g V}/{|\nabla_g V|_g}.$

Now, for the first integral on the right hand side of \eqref{schoenpoh}, since $\mathcal L_Zg=2\hess_g V$ we have
\begin{align*}\label{eqzb}
	\int_{M} \left\langle \mathrm{S}_g,\mathcal{L}_{Z}g\right\rangle d\mathcal{V}
	&= 2 \int_{M}V\big\langle\mathrm{S}_g,\Ric_g -\Lambda g + 2E^{\flat}\otimes E^{\flat} - |E|^2g\big\rangle\,d\mathcal{V}\\
	&= 2\int_{M}V|\mathrm{S}_g|^2 d\mathcal{V},
\end{align*}
where we have used \eqref{sev1} and the fact that $\langle\mathrm{S}_g,g\big\rangle=\mathrm{tr}_g\,\mathrm{S}_g = 0$.

It remains to calculate the integral on the boundary: 
\begin{eqnarray*}
\int_{\partial M}\mathrm{S}_g(\nabla_g V,N)\,d\mu&=&\int_{\partial M}\Big[ \Ric(\nabla_g V,N )-\frac{R_g}{3}\langle \nabla_g V,N \rangle\\
&& + \ 2\langle E,\nabla_g V\rangle\langle E,N \rangle - \frac{2}{3}|E|^2\langle \nabla_g V,N \rangle\Big]d\mu.
\end{eqnarray*}

By Lemma \ref{lemma:propert}, $E$ and $\nabla_g V$ are parallel along $\partial M$. Hence $\langle E,\nabla_g V \rangle^2=|E|^2|\nabla_g V|^2$ along $\partial M$. Using the Gauss equation we obtain
 
\begin{eqnarray}\label{eqbound}
\int_{\partial M}\mathrm{S}_g(\nabla_g V,N)\,d\mu &=& \sum_{i=1}^{\ell} k_i\int_{\partial_i M} \left(\frac{R_g}{3} -\Ric(N ,N ) - \frac{4}{3}|E|^2\right)d\mu_i\nonumber\\
&=& \sum_{i=1}^{\ell} k_i\int_{\partial_i M} \left(K_{\partial_i M} - \frac{R_g}{6} - \frac{4}{3}|E|^2\right)d\mu_i\nonumber\\ 
&=& \sum_{i=1}^{\ell} k_i\int_{\partial_i M} \left(K_{\partial_i M} - \frac{\Lambda}{3} - \frac{5}{3}|E|^2\right)d\mu_i\nonumber.
\end{eqnarray}
Hence, using the Gauss-Bonnet Theorem and \eqref{schoenpoh} we conclude that
\begin{align*}
2\pi\sum_{i=1}^{\ell} k_i\,\chi(\partial_i M) &= \int_{M}\biggl(|\mathrm{S}_g|^2 + \frac{2}{3}(\Lambda - |E|^2)|E|^2\biggr)V\,d\mathcal{V}\nonumber\\
& + \ \sum_{i=1}^{\ell}\int_{\partial_i M}k_i\left(\frac{\Lambda}{3}+|E|^2\right)d\mu_i.
\end{align*}
\end{proof}

%Boucher-Gibbons-Horowitz \cite{BGH} and Shen \cite{Shen} were able to show that the boundary $\partial M$ of a compact three-dimensional oriented static manifold $(M^3, g)$ with connected boundary and scalar curvature $R_g=6$  satisfies the inequality $|\partial M|\leq 4\pi$ with equality if, and only if, $(M^3, g)$ is the standard hemisphere, see also \cite{HMR} for a related result. Based in these results and  as an application of Proposition \ref{prop-sphere} we obtain the following rigidity result.
 
\begin{theo}\label{uniq}
Let $(M,g,V,E)$ be a compact electrostatic system, such that $V^{-1}(0) = \partial M = \cup_{i=1}^{\ell}\partial_i M$. Suppose $\sup_M |E|^2 < \Lambda$. Then, 
\begin{equation}\label{estcomparig0}
\sum_{i=1}^{\ell} k_i\left(\frac{\Lambda}{3}\,|\partial_i M| + \frac{16\pi^2 Q(\partial_i M)^2}{|\partial_i M|}\right) \leq 4\pi\sum_{i=1}^{\ell} k_i.
\end{equation}
Moreover, the equality holds if, and only if, $E \equiv 0$
and $(M,g)$ is isometric to the standard hemisphere of constant sectional curvature $\frac{\Lambda}{3}$.
\end{theo}
\begin{proof}
The assumption $\sup_M |E|^2 < \Lambda$ implies that $\Lambda > 0$, in particular $\Lambda + |E|^2 > 0$. So, by Theorem \ref{topology}, each $\partial_i M$ is a sphere. It follows from equation \eqref{prop-sphere} that
$$\sum_{i=1}^{\ell}\int_{\partial_i M}k_i\left(\frac{\Lambda}{3}+|E|^2\right)d\mu_i \leq 4\pi\sum_{i=1}^{\ell}k_i.$$
The proof of the area-charge estimate \eqref{estcomparig0} is then similar to the proof of Proposition \ref{mainineq}. 
 
 The equality case follows from the analysis of the equation 
$$|\mathrm{S}_g|^2 + \frac{2}{3}(\Lambda - |E|^2)|E|^2=0.$$ Since $\sup_M |E|^2 < \Lambda,$ we conclude that $|\mathrm{S}_g|^2=(\Lambda - |E|^2)|E|^2=0.$ Thus $|E|^2=0$ and $(M,g)$ is Einstein. Since the manifold has dimension $3$, it follows that $(M,g)$ has constant sectional curvature $\frac{\Lambda}{3}$. In particular $(M,g)$ has positive Ricci curvature, which implies that two distinct closed minimal surfaces must intersect (\cite[Section 2]{Fr}). Hence $\partial M$ has only one component, which is unstable (there is no two-sided closed stable minimal surface in a manifold of positive Ricci curvature). So, by Theorem \ref{topology}, $M$ is simply connected. Since $\partial M$ is totally geodesic we conclude that $(M,g)$ is isometric to the standard hemisphere.
\end{proof}

\appendix

\section{Existence and regularity of solutions of a Plateau type problem}\label{appendix}

Let $\Omega$ be a compact smooth domain in a complete Riemannian $3$-manifold $(M,g)$, such that $\partial\Omega$ is either empty or consists of closed mean convex surfaces. By Nash's theorem, there is $L \geq 3$ such that $(M,g)$ isometrically embeds in $\mathbb{R}^L$. Let $\mathcal{U}$ be a mean-convex domain in $\Omega$. Suppose that there exist a compact surface $S \subset \mathcal{U}$ such that $\gamma=\partial S \subset \partial\mathcal{U}$, and subsets $\{p_1,\cdots,p_m\}, \{q_1,\cdots,q_n\}$ of $\gamma$ satisfying:
\begin{enumerate}
\item $\{p_1,\cdots,p_m\}\cap\{q_1,\cdots,q_n\} = \emptyset$,
\item $S\setminus\{p_1,\cdots,p_m\}$ is embedded,
\item $S\setminus\{q_1,\cdots,q_n\}$ is smooth.
\end{enumerate}
We shall call {\it singularities} the points $p_1,\cdots,p_m,q_1,\cdots,q_n$.

We define
$$\mathcal{D}^k(M) = \{C^{\infty}\textrm{-} \ k\textrm{-}\textrm{forms} \ \omega\ \textrm{in}\ \mathbb{R}^L;\ \spt \ \omega \subset M\}$$
with the usual topology of uniform convergence of all derivatives on compact subsets.
We denote by ${\bf I}_k(M)$ the dual space of $\mathcal{D}^k(M)$, which is called the space of $k$-dimensional integral currents in $\mathbb{R}^L$ with support contained in $M$. The {\it mass} of $T \in {\bf I}_k(M)$, defined by
$${\bf M}(T)=\sup\{T(\phi): \phi \in \mathcal{D}^k(\mathbb{R}^L), ||\phi|| \leq 1\} \leq +\infty,$$
induces the metric ${\bf M}(S,T)={\bf M}(S-T)$ on ${\bf I}_k(M)$.
 Here $\mathcal{D}^k(\mathbb{R}^L)$ denotes  the space of smooth $k$-forms in $\mathbb{R}^L$ with compact support, and $||\phi||$ denotes the comass norm of $\phi$.

We define the class $\mathfrak{C}$ of admissible currents by
\begin{eqnarray*}
\mathfrak{C} = \{T \in {\bf I}_2(\Omega); \ \spt T \subset \mathcal{U} \ \mbox{and is compact}, \ \mbox{and} \ \spt \bigl(\partial T\bigr) = \gamma\},
\end{eqnarray*}
We want to minimize area in $\mathfrak{C}$, that is, we are looking
for $T \in \mathfrak{C}$ such that
\begin{equation}\label{var.prob}
{\bf M}(T) = \inf\{{\bf M}(\tilde{T}); \ \tilde{T} \in \mathfrak{C}\}.
\end{equation}
The proof of the existence of a minimizing current for the problem \eqref{var.prob} can be found in \cite{Fed}[5.1.6]. Let $T \in \mathfrak{C}$ be a solution. Then
\begin{eqnarray*}
{\bf M}(T) &\leq & {\bf M}(T + X) \label{min.prop},\\
\spt T &\subset & \mathcal{U}
\end{eqnarray*}
for any $X \in {\bf I}_2(\Omega)$ with compact support such that $\spt X \subset \mathcal{U}$ and $\partial X = 0$. 

Now, lets discuss the regularity of $T$. First, we need the following proposition.

\begin{prop}\label{edelen}
Consider $p\in \gamma$ and let $B_r(p)$ denote the geodesic ball of center $p$ and radius $r$ of $\Omega$. Consider $X \in {\bf I}_2(\Omega)$ such that $\spt X \subset B_r(p)$ and is compact, and $\partial X = 0$. If $T$ is a solution of \eqref{var.prob}, then there is a constant $C'$ such that for $r>0$ sufficiently small we have
\begin{equation}\label{almostminimizing}
{\bf M}_{B_r(p)}(T)\leq (1+C'r){\bf M}_{B_r(p)}(T + X).
\end{equation}
\end{prop}

\begin{proof}
To simplify the notation let us write $T$ to denote $T\cap B_r(p).$ Let $\tilde T \in {\bf I}_2(\Omega)$ be an integer multiplicity rectifiable current such that $\spt \tilde T \subset B_r(p)$ and is compact, and $\spt \bigl(\partial\tilde T\bigr)=\spt \bigl(\partial T\bigr)$. 

Consider $\mathcal{P}:B_r(p)\to B_r(p)\cap \mathcal{U}$ the projection onto $B_r(p)\cap \mathcal{U}$, i.e., for $q\in B_r(p)$, $\mathcal{P}(q)$ is the closest point to $q$ in $B_r(p)\cap \mathcal{U}$. Observe that for $r$ sufficiently small, $\mathcal{P}$ is well defined, almost everywhere smooth and Lipschitz. We can find a constant $C > 0$ such that
\begin{equation}\label{eq-der}
|D\mathcal{P}(q)|< 1 + Cr, \ \textrm{for a.e}\ q \in B_r(p).
\end{equation}

Let $X$ be as in the statement of the proposition. Denote $\tilde T = T + X$. Using inequality \eqref{eq-der} we get
\begin{equation*}\label{ineq.alm.min}
{\bf M}(T)\leq {\bf M}(\mathcal{P}(\tilde T)) \leq (1+Cr)^2 {\bf M}(\tilde T) \leq (1+C'r){\bf M}(\tilde T),
\end{equation*}
for some constants $C, C'$.
\end{proof}

\begin{theo}\label{last}
Let $T$ be a solution of \eqref{var.prob}. Then away from the singularities of $\gamma$, $T$ is supported in a connected oriented embedded minimal surface.
\end{theo}

\begin{proof}
By \cite[Theorem 4]{Whi2}, either $\spt T \subset \partial\mathcal{U}$ or $\spt T \cap\partial\mathcal{U} = \gamma$. In the first case, $\spt T\setminus\bigl(\{p_1,\cdots,p_m\}\cup\{q_1,\cdots,q_n\}\bigr)$ is automatically smooth and embedded. In the remaining of the proof we will suppose the second case happens. We can then apply the classical interior regularity theory (see \cite{Fed}) to prove that $\spt T \setminus\spt\partial T$ is a smooth embedded minimal surface.

Now, let $p$ be a point in the boundary away from the singularities. Since the regularity is a local question, we can assume that $\spt T \subset B_{r_0}(p)$, for some $r_0 > 0$, and this ball does not contain the singularities of $\gamma$. Since $\mathcal{U}$ is compact and mean convex, it satisfies a uniform exterior ball condition at $p$, i.e., there exists $r > 0$ such that, for any $q \in \partial \mathcal{U}$ there is a ball $B_r(y) \subset \Omega\setminus\mathcal{U}$ which is tangent to $\partial\mathcal{U}$ at $q$. Moreover it holds the almost-minimizing property \eqref{almostminimizing}. So we can argue as in the proof of Theorem $6$ in \cite{LiMe} to conclude the result.
\end{proof}

\end{document}